\documentclass[11pt]{article} %% Equations numbered as (1.1), (1.2) etc.

\usepackage{glaudo_en_pkg}
\usepackage[latin1]{inputenc}
\usepackage[english]{babel}
\usepackage{amsmath}
\usepackage{amsthm}
\usepackage{stmaryrd}
\usepackage{amssymb}

\usepackage{amsfonts}
\usepackage[OT2,OT1]{fontenc}

\usepackage{leftidx}
\usepackage{mathtools}
\usepackage{tikz}
\usetikzlibrary{matrix,arrows,decorations.pathmorphing}
\usetikzlibrary{quotes}
\usepackage{tikz-cd}

\newcommand{\N}{\ensuremath{\mathbb N}} % Natural numbers
\newcommand{\Z}{\ensuremath{\mathbb Z}} % Integers
\newcommand{\Q}{\ensuremath{\mathbb Q}} % Rationals
\newcommand{\F}{\ensuremath{\mathbb F}} % Mathbb F

\let\to\relax

\newcommand{\to}[1][]{\if\relax\detokenize{#1}\relax\rightarrow\else\xrightarrow{#1}\fi}

\newcommand{\defeq}{\ensuremath{\coloneqq}}% Equal in a definition.

\newcommand{\qz}{\ensuremath{\mathbb{Q}/\mathbb{Z}}}

\DeclareFontFamily{U}{wncy}{}
\DeclareFontShape{U}{wncy}{m}{n}{<->wncyr10}{}
\DeclareSymbolFont{mcy}{U}{wncy}{m}{n}
\DeclareMathSymbol{\Sha}{\mathord}{mcy}{"58}

\renewcommand{\theta}{\vartheta}

\DeclareMathOperator{\Hom}{Hom}

\DeclareMathOperator{\im}{Im}
\DeclareMathOperator{\Gal}{Gal}

\DeclareMathOperator{\Spec}{Spec}
\DeclareMathOperator{\Pic}{Pic}
\DeclareMathOperator{\tr}{tr}

\DeclareMathOperator{\Br}{Br}

\DeclareMathOperator{\Ker}{Ker}
\DeclareMathOperator{\Coker}{Coker}
\DeclareMathOperator{\Gr}{Gr}

\DeclareMathOperator{\res}{res}
\DeclareMathOperator{\inv}{inv}

\newcommand{\s}{\setminus}

\newcommand{\cF}{\mathcal{F}}

\newcommand{\A}{\mathbb{A}}
\newcommand{\G}{\mathbb{G}}

\renewcommand{\P}{\mathbb{P}}

\newcommand{\ok}{\overline{k}}
\newcommand{\oK}{\overline{K}}
\newcommand{\oF}{\overline{F}}

\newcommand{\sm}{\text{sm}}

\newcommand{\et}{\acute{e}t}

\newcommand{\tpsi}{\psi}

\newtheorem*{Thm:lifting}{Theorem \ref{Thm:lifting}}
\newtheorem*{Cor:ObstructionIntro}{Corollary \ref{Cor:ObstructionIntro}}

\title{Ramified descent and transcendental Brauer--Manin obstruction}

\author{Julian Lawrence Demeio}

\begin{document}
	
	\maketitle

\begin{abstract}
	We investigate the following problem. Given a smooth geometrically connected variety $X$ defined over a number field $K$, and an \'etale torsor $V \to U$ over a Zariski-open $U$ of $X$, which adelic points of $X$ can be approximated by adelic points that lift to a (twist of a) $V$? The question has long been investigated in the literature when $U=X$, but less so in the general case. We introduce a Brauer--Manin obstruction to the problem, and provide an example where this obstruction is non-trivial and purely transcendental. This answers in the negative a question posed by Harari at a 2019 workshop. Our example is also an explicit example of a  non-trivial transcendental Brauer--Manin obstruction on a smooth compactification of a quotient $SL_n/G$, with $G$ constant metabelian.
\end{abstract}

\section{Introduction}

Descent theory has long been used to understand how rational points $X(K)$ of a smooth complete variety $X$ defined over a number field $K$ are distributed in the adelic points $X(\A_K)$. It was first developed for proper varieties by Colliot-Th\'el\`ene and Sansuc \cite{CTS87} \cite{Skorobogatov}, and it was later extended to open varieties by Harari and Skorobogatov \cite[Chapter 6]{opendescent}. We investigate the matter of ``ramified descent'', i.e.\ the behaviour of open descent theory under compactification, and answer a question of Harari on this topic.

For a torsor $\lambda:V \to U$ under a group of multiplicative type $M/K$, the {\bf descent set} is defined to be the set of those adelic points of $U$ that lift to adelic points of a $K$-twist of $V$:
\[
U(\A_K)^{\lambda}=\bigcup_{\sigma \in H^1(K,M)} \lambda_{\sigma}(V_{\sigma}(\A_K)).
\]
It follows from open descent theory \cite[Proposition 3.1]{opendescent}  that this set may be described in terms of a(n algebraic) Brauer--Manin obstruction: i.e.\ there exists a subgroup $\Br_{\lambda}U\subseteq \Br_1 U$ such that 
\begin{equation}\label{Eq:HS}
    U(\A_K)^{\lambda}=U(\A_K)^{\Br_{\lambda}U}.
\end{equation}
Let now $X$ be a smooth compactification of $U$, and $X(\A_K)^{\lambda}$ be the adelic closure of $U(\A_K)^{\lambda}$ in $X(\A_K)$. The first result of this paper is that $X(\A_K)^{\lambda}$ provides an obstruction to the Hasse Principle and weak approximation for $X$:
\begin{theorem}
    The inclusion $\overline{X(K)} \subseteq X(\A_K)^{\lambda}$ holds.
\end{theorem}
The proof we present, quite compact, is due to Olivier Wittenberg, whom the author thanks profoundly. (The original proof that the author had in mind was much more involved.) 

We are mainly interested in this paper in the case where $M$ is finite, in which case we call $X(\A_K)^{\lambda}$ the ``ramified descent set'' (the adjective ``ramified'' indicates that the relative normalization of $U$ in $V$ is allowed to be ramified). (See also Section \ref{Sec:Setting} for an alternative definition of the ramified descent set.)

In a 2019 workshop, Harari formulated a question to investigate how \eqref{Eq:HS} behaves under ``compactification''. We present his question in Section \ref{SSec:Reformulation} (Question \ref{QHarari}), but a special, yet exemplative, case may be reformulated as follows (see Proposition \ref{Prop:reformulation}):
\begin{question}[Harari]\label{Q11}
    Assume that $\Br X/\Br_0 X$ and $M$ are finite. Does the identity $X(\A_K)^{\lambda}= X(\A_K)^{\Br_{\lambda}U \cap \Br X}$ hold?
\end{question}
Note that the  inclusion $X(\A_K)^{\lambda}\subseteq  X(\A_K)^{\Br_{\lambda}U \cap \Br X}$ follows from \eqref{Eq:HS}. 
Harari's question was motivated by the fact that the analog of Question \ref{Q11} has a positive answer when $M$ is a torus (see \cite[Proposition 3.1]{BMS14}). Moreover, when $M=\mu_n$ is cyclic, and some mild ramification assumptions are satisfied, then a positive answer follows from a result of Colliot-Th\'el\`ene and Skorobogatov \cite[Theorem 14.2.25]{BGbook} (see Appendix \ref{Sec:Other} for more details).

We answer Question \ref{Q11} negatively. In order to so, we introduce in Section \ref{Sec:Obstruction} a new Brauer subgroup $\Br_{\lambda}^{ram}X \subseteq \Br X$, defined as the intersection $\Br_{\lambda}^{ram}U \cap \Br X$, where $\Br_{\lambda}^{ram}U$ is the image of the composition 
\[
H^2(\Gamma_M, \oK^*) \to H^2(\Gamma_M, \oK[V]^*) \to[\check{C}_{\mathfrak l}] H^2(U,\G_m)=\Br U,
\]
where $\check{C}_{\mathfrak l}$ is the \v{C}ech-to-\'etale map associated to the $V_{\oK} \to U$, which is a profinite torsor under the constant profinite group $\Gamma_M  \coloneqq M(\oK) \rtimes \Gamma_K$ (see Section \ref{Sec:Obstruction} for more details). In the special case where both $\Pic V$ and $\oK[V]^*/\oK^*$ vanish, one has $\Br_{\lambda}^{ram}X=\Ker(\Br X \to \Br V_{\oK})$, see Remark \ref{Rmk:Piceq0}. (In Example \ref{ExTutto0} we provide some examples where this vanishing happens.)
One verifies that: 
\begin{proposition}\label{Prop2}
    The inclusion $X(\A_K)^{\lambda} \subseteq X(\A_K)^{\Br_{\lambda}^{ram}X}$ holds.
\end{proposition}
% The definition of $\Br_{\lambda}^{ram}X$ is involved in the general case, but
The group  $\Br_{\lambda}^{ram}X$ contains $\Br_{\lambda}U \cap \Br X$, see Proposition \ref{Prop:InclusionBraBr}. However, the former may in general be a bigger group and provide a bigger Brauer--Manin obstruction. We provide some explicit families where this is indeed the case in Section \ref{Sec:GBsp}, where we prove:
\begin{proposition}\label{Prop:gbsp}
    Let $K$ be a number field and $e$ be a natural number such that $\mu_e \subset K^*$. If $H$ is a constant metabelian finite subgroup of $SL_n$ of exponent $e$, $H^{ab}$ is its abelianization, and $\lambda$ is the $H^{ab}$-torsor $SL_n/[H,H] \to SL_n/H$, then $\Br_{\lambda}^{ram}X=\Br X$.
\end{proposition}

\begin{theorem}\label{Thm:gbsp}
    For every number field $K$ and every prime $p \geq 5$ such that $\mu_p \subset K^*$, there exists a constant nilpotent metabelian finite group $H$ of exponent $p$ such that, for any embedding $H \hookrightarrow SL_{n,K}$, letting $X$ be a smooth compactification of $SL_{n,K}/H$, we have $\Br_a X= 0$ and
    \[
    X(\A_K)^{\Br X} \neq X(\A_K).
    \]
\end{theorem}

Here $\Br_aX$ denotes the algebraic Brauer group of $X$ modulo constants. 
Combining Proposition \ref{Prop:gbsp}, Theorem \ref{Thm:gbsp},  and Proposition \ref{Prop2} we obtain the sought negative answer to Harari's question: indeed for $X$ as in the theorem, we have that $\Br_{\lambda}U \cap \Br X \subseteq \Br_1X$ is constant as $\Br_aX=0$, while $X(\A_K)^\lambda \subseteq X(\A_K)^{\Br_{\lambda}^{ram} X} =X(\A_K)^{\Br X} \neq X(\A_K)$. 

Incidentally, Theorem \ref{Thm:gbsp} appears to be only the second known example of transcendental obstruction to weak approximation for quotients $SL_n/H$, or, in other words (see \cite[Sec.\  1.2]{Harari2007}) to the Grunwald Problem for a finite group $H$. The first such example was obtained by Demarche, Lucchini and Neftin in \cite[Theorem 1.2]{DLN17}. However, unlike in {\em loc.cit.}, where the existence of a transcendental obstruction is proven non-constructively, we show the non-triviality of the transcendental Brauer--Manin pairing by computing it explicitly. 

\paragraph{Structure of the paper}

In Section \ref{Sec:Notation} we settle our notation. In Section \ref{Sec:Setting} we prove some basic facts about the ``descent set'' 
 $X(\A_K)^{\lambda}$ for a finite $M$. Since it requires no further effort, we replace $M$ here with a general finite group scheme $G/K$ (not necessarily commutative). In the same section, we prove that $X(\A_K)^{\lambda}$ provides an obstruction to Hasse Principle and weak approximation on the whole $X$ (see Corollary \ref{Cor:ObstructionIntro}). 
 In Section \ref{SSec:Reformulation}, we formulate Harari's question precisely and show its relation with Question \ref{Q11}.
 In Section \ref{Sec:Obstruction}, we introduce the Brauer subgroup $\Br_{\lambda}^{ram}X$, prove that it contains $\Br_{\lambda}U \cap \Br X$ and that it obstructs $X(\A_K)^{\lambda}$. 
 In Section \ref{Sec6}, we prove Proposition \ref{Prop:gbsp} and Theorem \ref{Thm:gbsp}. We do so by explicitly computing the unramified Brauer--Manin pairing on $SL_n/H$ for $H$ nilpotent metabelian of odd prime exponent. This explicit computation extends earlier work of Bogomolov \cite[Sec.\ 5]{BogomolovMumbai}. 
 Appendix \ref{Sec:Elementary} contains some elementary lemmas that are used in Section \ref{Sec:GBsp}. Appendix \ref{Sec:Other} talks briefly about other already existing works containing the idea of ``ramified descent''.

\section{Notation}\label{Sec:Notation}

\paragraph{Fields}

Unless specified otherwise, $k$ will always denote a field of characteristic $0$ and $K$ a number field.

$M_K$ (resp.~$M_K^f, M_K^{\infty}$) denotes the set of (non-archimedean, archimedean) places of $K$.

For a place $v \in M_K$ (resp.~$v \in M_K^f$), $K_v$ (resp.~$O_v$) denotes the $v$-adic completion of $K$
(resp.~the $v$-adic integers).

$\A_K$ (resp.~$\A_K^S$, for a subset $S \subset M_K$) denotes the topological ring of adeles of $K$ (resp.~$S$-adeles), i.e. the topological ring $\prod'_{v \in M_K} K_v$ (resp.~$\prod'_{v \in M_K \setminus S} K_v$), the restricted product being on $O_v \subseteq K_v$.

For a finite subset $S \subseteq M_K$, $K_S$ denotes the product $\prod_{v \in S} K_v$. We let $K_{\Omega}$ denote the product $\prod_{v \in M_K} K_v$.

For a Galois extension $L/K$, $\Gal(L/K)$ denotes the Galois group of the extension. For a field $k$ with algebraic closure $\overline{k}$, $\Gamma_k\defeq \Gal(\overline{k}/k)$.

\paragraph{Duals}

For a group $M$ of multiplicative type over a field $k$ (i.e. a commutative group scheme which is an extension of a finite group scheme by a torus), $M' \defeq Hom(M, \G_{m})$ denotes its Cartier dual.

For a torsion abelian group $A$, $A^D$ denotes the profinite abelian group $\Hom(A,\Q/\Z)$ endowed with the compact-open topology. If $A$ is a profinite abelian group, $A^D$ denotes the torsion group $\Hom_{cont}(A,\Q/\Z)$, where $\Q/\Z$ is endowed with its discrete topology. By Pontryagin duality, if $A$ is torsion or profinite, there is a canonical isomorphism $A \cong(A^D)^D$. 

\paragraph{Geometry}

All schemes appearing in this paper are separated. We always tacitly assume this throughout the paper.

A \emph{variety} $X$ over a field $k$ is an integral scheme of finite type over a field $k$.

For a $k$-scheme $X$, we denote the residue field of a point $\xi \in X$ by $k(\xi)$. We denote the base change $X_{\ok}$ by $\overline{X}$.

%\paragraph{Covers and $G$-covers}

%A morphism $\psi:Y \rightarrow X$ is a \emph{cover} if $X$ and $Y$ are normal (we do not include integrality in our definition of normality), $\psi$ is finite, $X$ is integral and every connected component of $Y$ surjects onto $X$.
%If $G$ is a finite \'etale group scheme over a perfect field $F$, a $G$-\emph{cover} $\psi:Y \to X$ is a cover where both $X$ and $Y$ are $F$-schemes and such that there is an $X$-invariant $G$-action $Y \times_F G \to Y$ such that there is a non-empty open subscheme $U \subset X$ over which $\psi$ is an (\'etale) $G$-torsor.

\paragraph{Groups and torsors}

Group actions are assumed to be right actions unless specified otherwise. 

Let $S$ be a scheme, $G$ be a group scheme over $S$ and $X$ be an $S$-scheme. A right $G$-\emph{torsor} over $X$ is an $X$-scheme $Y \to X$, endowed with a $G$-action $m:Y \times_S G \to Y$ that is $X$-equivariant (i.e. such that the composition $Y \times_S G \to[m] Y \to X$ is equal to the composition $Y \times_S G \to[pr_1] Y \to X$) and such that there exists an \'etale covering $X' \to X$ and an $X'$-isomorphism $X' \times_X Y  \cong X' \times_X G$ that is $G$-equivariant. 

For an abstract group $N$, and a scheme $S$ (resp.\ a field $F$), we denote by $N_S$ (resp.\ $N_F$) the $S$-scheme (resp.\ $F$-scheme) $\sqcup_{n \in N} S$, endowed 
with its natural $S$(resp.\ $F$)-group scheme structure. 
If $X$ is an $S$-scheme, a torsor $Y \to X$ under an abstract group $G$ is a torsor under the constant group $G_S$.

If $G/k$ is an algebraic group, and $k \subseteq F$ is a field extension, we use the notation $H^i(F,G)$ (with $i \in \N$ and $i=0,1$ if $G$ is not commutative) to denote the cohomology group/set $H^i(\Gamma_F,G(\overline{F}))=Z^i(\Gamma_F,G(\overline{F}))/B^i(\Gamma_F,G(\overline{F}))$ (where $B^i(\Gamma_F,G(\overline{F}))$ is a subgroup when $G$ is commutative and is just an equivalence relation otherwise). 

If $G$ is not commutative the set of cocycles $Z^1(\Gamma_{F},G(\overline{F}))$ is the one of non-abelian ($1$-)cocycles, i.e. those functions $g_{\sigma}:\Gamma_F \to G(\overline{F})$ that satisfy 
$g_{\sigma \tau} = g_{\sigma}\leftidx{^\sigma}g_{\tau} $. The set $H^1(\Gamma_F,G(\overline{F}))$ is the  quotient of $1$-cocycles by the equivalence relation $B^1(\Gamma_F,G(\overline{F})): g_{\sigma} \sim  g'_{\sigma}$ if there exists $g \in G(\overline{F})$ such that $g'_{\sigma}=g^{-1}g_{\sigma}\leftidx{^\sigma}g$.
Note that these cocycles correspond to (left) $G$-torsors through the standard correspondence \cite[p.18, 2.10]{Skorobogatov}. 

If $\xi \in Z^1(K,G)$, we use the notation $G^{\xi}$ to denote the inner twist of $G$ by $\xi$, and $G_{\xi}$ to denote the left principal homogeneous space of $G$ obtained by twisting $G$ by the cocycle $\xi$. This twist is naturally endowed with a right action of $G^{\xi}$. See \cite[p. 12-13]{Skorobogatov} for more details on these constructions.

If $X$ is a quasi-projective $k$-scheme endowed with a $G$-action, and $\xi \in Z^1(k,G)$, we use the notation $X_{\xi}$ to denote the twisted quasi-projective $k$-scheme $(X \times_k^G G_{\xi})$. (We refer the reader to \cite[p. 20]{Skorobogatov}, \cite[Sec.\ I.5.3]{CohGaloisienne} and \cite[Sec.\ III.1.3]{CohGaloisienne} for the existence of the twist and immediate properties of the twisting operation). The $k$-scheme $X_{\xi}$  is naturally endowed with a $G^{\xi}$-action.
We recall that there always exists  a $G \times_k \bar{k}$-equivariant isomorphism $X_{\xi} \times_k \ok \cong X \times_k \ok$. If $X'$ is another $k$-scheme and $\psi:X \to X'$ is a $G$-invariant morphism (i.e.\ $G$-equivariant when we endow $X'$ with the trivial action), we denote by $\psi_{\xi}:X_{\xi} \to X'$ the twisted form of $\psi$ by $\xi$.

If $X$ is endowed with a left $G$-action we may still do the twisting operations, by taking the corresponding right action, using the canonical isomorphism $G \cong G^{op}, g \mapsto g^{-1}$.

\paragraph{Equivariant commutative diagrams}

Let $S$ be a scheme. For $S$-group schemes $G_1, G_2$, a (usually implicit) homomorphism $G_1 \to G_2$, and torsors $Z_1\to[G_1] W_1, Z_2\to[G_2] W_2$, a diagram
\begin{equation}\label{EqDCommaction}
\begin{tikzcd}
Z_1 \arrow[d, "G_1"] \arrow[r] & Z_2 \arrow[d, "G_2"] \\
W_1 \arrow[r]                  & W_2                 
\end{tikzcd}
\end{equation}
{\em commutes} if the underlying diagram is commutative and $Z_1 \to Z_2$ is $(G_1\to G_2)$-equivariant. 

%Equivalently, if the map $(Z_1,G_1) \to (Z_2,G_2)$ is a morphism in the category of torsors over $S$. 

\paragraph{Category of torsors} Let $S$ be a scheme, and $X$ an $S$-scheme. The category of torsors over $X$ with base-scheme $S$ is the category whose objects are pairs $(Y,G)$, where $G$ is an $S$-group scheme, and $Y$ is a $G$-torsor over $X$, and whose morphisms are pairs $(Y_1 \to Y_2,G_1 \to G_2)$, where $G_1 \to G_2$ is a homomorphism, and $Y_1 \to Y_2$ is a $(G_1\to G_2)$-equivariant morphism. 

\paragraph{Profinite (\'etale) torsors under constant profinite groups} 
Let $S$ be a scheme, $X$ an $S$-scheme, and $G$ be a(n abstract) profinite group. A {\em profinite torsor} over $X$ under $G$ is an $X$-scheme $Y \to X$, endowed with a $G$-action $m:Y \times G \to Y$ that is $X$-equivariant, and such that there exists an inverse system $(Y_i,G_i), i \in I$ in the category of torsors over $X$ with each $G_i$ finite, and such that there exist an isomorphism $\psi: G \to[\sim] \lim\limits_{\leftarrow} G_i$ and a $\psi$-equivariant isomorphism $Y \cong \lim\limits_{\leftarrow} Y_i$. (Note that the inverse limit of the $Y_i$ exists in the category of schemes, as each $Y_i$ is the relative spectrum of a finite \'etale $\mathcal O_X$-algebra $\mathcal O_{Y_i}$, and their inverse limit may be realized as the relative spectrum of the finite \'etale ind-algebra $\lim\limits_{\to}\mathcal O_{Y_i}$ \cite[Tag 01YV]{stacks-project}.)

When in addition $Y$ is connected, we say that $Y \to X$ is a {\em profinite \'etale Galois cover}, see \cite[Remark 2.21(b)]{LECcompleto}. See also \cite[Sec.\ 2]{vwprofinite} for more details and an alternative definition. %\jul{Puoi provare anche SGA1, V.5}

\paragraph{Brauer group}

Recall that the Brauer group of a scheme $X$ is defined to be the \'etale cohomology group $H^2_{\et}(X,\G_m)$. When $X$ is a variety defined over a number field $K$, this provides an obstruction, known as \emph{Brauer-Manin obstruction}, to local-global principles, in the following sense. There is a {\em Brauer-Manin }pairing:
\[
X(\A_K) \times \Br X \to \Q/\Z,
\]
sending $((P_v)_{v \in M_K},B)$ to $((P_v)_{v \in M_K},B)_{BM}:=\sum_v \inv_v B(P_v)$, where $\inv_v:H^2(\Gamma_{K_v}, \overline{K_v}^*) \to \Q/\Z$ is the usual invariant map (see e.g. \cite[Thm 8.9]{HarariBook} for a definition). Whenever $B \in  \im \Br K$ or $(P_v)_{v \in M_K} \in X(K)$ (diagonally embedded in $X(\A_K)$), $((P_v)_{v \in M_K},B)_{BM}=0$ by the Albert-Brauer-Hasse-Noether theorem (see  \cite[Sec.\ 5]{Skorobogatov}). It follows that $X(K)$ is a subset of
\[
X(\A_K)^{\Br X} := \{(P_v)_{v \in M_K} \in X(\A_K) \mid ((P_v)_{v \in M_K},B)_{BM}=0 \text{ for all } B \in \Br X \}.
\]

For a geometrically integral scheme $X$ over a field $F$, we adopt the usual notation $\Br_1 X:=\Ker (\Br X \to \Br X_{\oF})$ and $\Br_0X := \im (p^*: \Br F \to \Br X)$, where $p:X \to \Spec F$ denotes the structural morphism.

When $X$ is smooth and integral over $F$, and $U \subseteq X$ is an open subscheme, we identify, with a slight abuse of notation, the injective \cite[Thm.\  3.5.5]{BGbook} pullback $\Br X \to \Br U$ with an inclusion $\Br X \subseteq \Br U$ ($\subseteq \Br k(X)$). We say that $\beta$ is \emph{unramified} if $\beta \in \Br X^c$, for one (or, equivalently, all  \cite[Prop.\ 3.7.10]{BGbook}) smooth compactification(s) $X^c$ of $X$. We denote the subgroup of unramified elements by $\Br_{ur} (k(X))$ or $\Br_{ur}X$.

\paragraph{Cohomology}

For a scheme $X$ and an \'etale abelian sheaf $\cF$ on $X$, $H^n(X,\cF), n \geq 0$ denotes the \'etale cohomology group $H^n_{\et}(X,\cF)$.

\paragraph{Map from \v{C}ech cohomology to \'etale cohomology}

Let $U$ be a scheme, $\cF$ an \'etale sheaf on $U$, and $\phi: V \to U$ an \'etale cover. We may think of $\phi$ as an \'etale covering $\mathcal{U}$ of $U$ made of a single cover: $\mathcal{U}=\{V \to U\}$.  Recall that to such a covering we may naturally associate its \v{C}ech cohomology groups $\check{H}^n(V/U,\cF), n \geq 0$ \cite[Sec.\ III.2]{LECcompleto}. There are natural \v{C}ech-to-\'etale morphisms:
\begin{equation}\label{Cechtoet}
    \check{C}_{\phi} :\check{H}^n(V/U,\cF) \to H^n(U,\cF)
\end{equation}
for each $n \geq 0$. (As edge maps of the first spectral sequence in \cite[Prop.\ III.2.7]{LECcompleto}).

When $\phi:V \to U$ is an \'etale torsor under the constant finite group $G$, there are natural identifications  \cite[Ex.\ III.2.6]{LECcompleto} (technically {\em loc.cit.} is formulated for Galois covers, equiv.\ for {\em connected} torsors under constant finite groups, but the connectedness assumption is never used):
\begin{equation}\label{Grouptocech}
    \check{H}^n(V/U,\cF)=H^n(G,\cF(V)),
\end{equation}
where the latter denotes group cohomology. Under these identifications, \eqref{Cechtoet} becomes:
\begin{equation}\label{Cechtoet2}
    \check{C}_{\phi} :H^n(G,\cF(V)) \to H^n(U,\cF).
\end{equation}
If $\phi:V \to U$ is a torsor under a profinite constant group $G$, then one may define natural maps $\check{C}_{\phi} :H^n(G,\cF(V)) \to H^n(U,\cF), n \geq 0,$
as the colimit of \eqref{Cechtoet2} on all quotients $V/H \to U$ by open subgroups $H$ of $G$. See \cite[Rmk.\ III.2.21(b)]{LECcompleto} for more details. 

\begin{remark}\label{Rmk:trivialtorsorCech}
    When $Y=X \times G$ (i.e.\ $\phi$ is the trivial torsor) with $G$ finite, then the trivial covering $\mathcal{U}'=\{U \to U\}$ refines $\mathcal{U}=\{V \to U\}$ via the morphism $U= U \times \{e\} \hookrightarrow U \times G$. Thus $\check{C}_{\phi}:H^n(G,\cF(V)) \to H^n(U,\cF)$ factors  through $H^n(U/U,\cF(U))$. The latter is $0$ for $n\geq 1$, and thus $\check{C}_{\phi}=0$ for $n\geq 1$.
\end{remark}

\section{Descent set}\label{Sec:Setting}

\subsection{Definition}

\paragraph{Descent set for torsors}

Let $K$ be a number field, $G/K$ a finite group scheme, $p:U \to \Spec K$ a smooth geometrically connected variety over $K$, and $\lambda:V \to U$ a $G$-torsor. 

%To recall the definition of the descent set of a torsor, let us first recall the definition (and immediate properties) of the twist of a torsor by a cocycle.

For every $\xi \in H^1(K,G)$, there exists a twisted form $\lambda_{\xi}:V_{\xi} \to U$ of the torsor $\lambda$. This is a torsor under the twisted form $G^{\xi}$ of $G$. The class $[\lambda_{\xi}] \in H^1(U,G^{\xi})$ is given by the image of $[\lambda] \in H^1(U,G)$ under the well-known isomorphism  \cite[p.20, 21]{Skorobogatov}
$$
H^1(U,G) \to  H^1(U,G^{\xi}), [V] \mapsto [V_{\xi}].
$$ 
When $G$ is commutative, we have that $G^{\xi} = G$, and the morphism $H^1(U,G) \to  H^1(U,G), [V] \mapsto [V_{\xi}]$ becomes $[V] \mapsto [V] - p^*[\xi]$.

Recall that the descent set $U(\A_K)^{\lambda}$ associated to $\lambda$ is defined as:
\begin{equation}\label{Eq:Obstr_integral}
{U(\A_K)^{\lambda}}\defeq {\bigcup_{\xi \in H ^1(K,G)}\lambda_{\xi}(V_{\xi}(\A_K))} \subseteq U(\A_K),
\end{equation}

This is adelically closed in $U(\A_K)$ \cite[Prop. 6.4]{demarche} and contains ${U(K)}$ \cite[Sec.\ 5.3]{Skorobogatov}.

\paragraph{Compactifying the descent set}

Let $X$ be a smooth compactification of $U$. Recall from the introduction:

\begin{definition}
    The {\em ramified descent set} for $\lambda$ is:
    \[
    X(\A_K)^{\lambda}\coloneqq \overline{\bigcup_{\xi \in H^1(K,G)}\lambda_{\xi}(V_{\xi}(\A_K))}.
    \]
\end{definition}

The closure denotes the adelic closure in $X(\A_K)$. 
Let $\psi:Y \to X$ be the relative normalization of $X$ in $V$. By the universal property of the relative normalization, the $G$-action on $V$ extends to a $G$-action on $Y$. Let $\nu:Y^{\sm} \to Y$ be a $G$-equivariant desingularization of $Y$ \cite{strong_resolution}, and let $\psi^{\sm}$ be the composition $\tpsi\circ\nu:Y^{sm} \to X$. The following lemma provides alternative descriptions of the ramified descent set.

\begin{lemma}\label{Lem:equivdefs}
    We denote by $\psi_{\xi}$ (resp.\ $\psi^{\sm}_{\xi}$) the twisted forms of $\psi$ (resp.\ $\psi^{\sm}$) by $\xi \in H^1(K,G)$. The following sets coincide:
	\begin{enumerate}
		\item $\overline{U(\A_K)^{\lambda}}$,
		\item $\overline{\bigcup_{\xi \in H ^1(K,G)}
			\psi^{\sm}_{\xi}(Y^{\sm}_{\xi}(\A_K))}$,
		\item $\overline{\bigcup_{\xi \in H^1(K,G)}\psi_{\xi}(Y^{reg}_{\xi}(\A_K))},$ where $Y^{reg}$ is the open subscheme of regular points of $Y$.
	\end{enumerate}
\end{lemma}

The closures denote adelic closures in $X(\A_K)$.

\begin{proof}
	We first prove that i and ii coincide. Note that $V' \defeq \nu^{-1}(V) \to[\nu] V$ is an isomorphism since $V$ is regular.
	We have that:
	\begin{equation}\label{Eqtant}
	\overline{\bigcup_{\xi}\lambda_{\xi}(V_{\xi}(\A_K))} =\overline{\bigcup_{\xi} \psi^{\sm}_{\xi}(V'_{\xi}(\A_K))} = \overline{\bigcup_{\xi} \overline{\psi^{\sm}_{\xi}(V'_{\xi}(\A_K))}}  =\overline{\bigcup_{\xi} \psi^{\sm}_{\xi}(\overline{V'_{\xi}(\A_K)})} = \overline{\bigcup_{\xi} \psi^{\sm}_{\xi}(Y^{\sm}_{\xi}(\A_K))},
	\end{equation}
	where the union is over ${\xi \in H ^1(K,G)}$ everywhere, and in the third term, $\overline{V'_{\xi}(\A_K)}$ denotes the closure in $Y^{\sm}_{\xi}(\A_K)$. The first two identities are immediate, the third follows from the properness of $\psi^{\sm}_{\xi}$, and the fourth holds because $V'_{\xi}(\A_K)$ is dense in $Y^{\sm}_{\xi}(\A_K)$ ($Y^{\sm}_{\xi}$ is smooth, so this follows from \cite[Thm.\  10.5.1]{BGbook}). This proves that i and ii coincide. They also coincide with iii, since this is contained between the left and right hand side of \eqref{Eqtant}.
\end{proof}

\begin{remark}\label{Rmk:unamb}
	\begin{itemize}
		\item Since iii is independent of the choice of $U$ and $Y^{\sm}$, the lemma shows that ii is as well, and i is too in the sense that $X(\A_K)^{\lambda}$ only depends on the generic fiber $\lambda|_{\Spec K(X)}:\Spec K(V) \to \Spec K(X)$, viewed as a $G$-torsor, of the torsor $\lambda$. 
		\item No conflict of notation on $X(\A_K)^{\lambda}$ arises with \eqref{Eq:Obstr_integral} when $U=X$, as in this case $U(\A_K)^{\lambda} = X(\A_K)^{\lambda}$ is closed in $X(\A_K)$ by \cite[Prop. 6.4]{demarche}.
	\end{itemize}
\end{remark}

\begin{warning*}
	Note that, as the continuous map $U(\A_K) \hookrightarrow X(\A_K)$ is not a topological immersion, the set $X(\A_K)^{\lambda} \cap U(\A_K)$ might in general very well be bigger than $U(\A_K)^{\lambda}$. The reader may verify that in the example given in Section \ref{Sec:GBsp} this is exactly the case.
\end{warning*}

\paragraph{Setting}

From now on we {fix}, until Section \ref{Sec:Obstruction} (included) a number field $K$, a finite group scheme $G/K$, a $G$-torsor $\lambda:V \to U$ over a geometrically integral smooth $K$-variety $p:U \to \Spec K$, and a smooth compactification $X$ of $U$.

\subsection{Obstruction to adelic density of rational points on $X$}\label{Sec:ClosureofXK}

%In this subsection, we prove that $\overline{X(K)} \subseteq X(\A_K)^{\lambda}$ (the closure being in $X(\A_K)$). See Corollary \ref{Cor:Theorem11}. As previously remarked, one easily sees that $\overline{U(K)} \subseteq X(\A_K)^{\lambda}$ (the closure again being in $X(\A_K)$), however, in general, one may well have that $\overline{U(K)}$ is strictly smaller than $\overline{X(K)}$. This is the case, for instance, when $U(K)= \emptyset$ while $X(K)\neq  \emptyset$.

Let, as above, $\psi:Y \to X$ be the relative normalization of $X$ in $V$, $\nu:Y^{\sm} \to Y$ be a $G$-equivariant desingularization of $Y$, and $r$ be the composition $\tpsi\circ\nu:Y^{sm} \to X$. 
%In this subsection we prove Theorem \ref{Thm:lifting} and its Corollary \ref{Cor:ObstructionIntro}. We recall the statements:

\begin{theorem}\label{Thm:lifting}
    The inclusion $X(K) \subseteq \bigcup_{\xi \in H^1(K,G)} r_{\xi}(Y_{\xi}(K))$ holds.
\end{theorem}

Combining Theorem \ref{Thm:lifting} with Lemma \ref{Lem:equivdefs}.ii, we deduce:

\begin{corollary}\label{Cor:ObstructionIntro}
	The inclusion $\overline{X(K)} \subseteq X(\A_K)^{\lambda}$ holds.
\end{corollary}

The following proof of Theorem \ref{Thm:lifting} is due to Olivier Wittenberg, who kindly suggested a proof that is much simpler than the previous one the author had.

\begin{proof}[Proof of Theorem \ref{Thm:lifting} \emph{(}\emph{Olivier} \emph{Wittenberg}\emph{)}]
	
    Let $d = \dim X$, $P \in X(K)$ be a rational point, and $u_1,\ldots,u_d \in \mathcal{O}_{X,P}$ be a regular system of parameters at $P$. Let $C \subseteq X$ be the Zariski-closure of the curve $u_2=\cdots = u_{d}=0$. Since $K$ is infinite, after a linear change of coordinates of $u_1,\ldots,u_d$, we may assume that $C$ is not contained in $D \coloneqq X \s U$. 
	
	Note that $C$ is smooth at $P$. Choosing a local parameter $t$ for $C$ at $P$, we get a morphism $\Spec K [[t]] \to C$ that sends the special point $t=0$ to $P$. This morphism induces a morphism $\Spec K((t)) \to X$, whose set-theoretic image is the generic point of $C$. In particular, by construction of $C$, it belongs to $U$. Hence the $G$-torsor $V \to U$ gives a class in $H^1({K ((t))}, G)$, which  we may push to $H^1({K((t^{\frac{1}{\infty}}))}, G)$.
	
	The inclusion $K \subseteq K((t^{\frac{1}{\infty}}))$ induces an identification $\Gamma_{K ((t^{\frac{1}{\infty}}))} = \Gamma_K$ (this follows from the algebraic-closedness of $\overline{K}((t^{\frac{1}{\infty}}))$ \cite[Chapter IV, Prop.\  8]{LocalFields}), and hence an identification 
	$H^1({K((t^{\frac{1}{\infty}}))}, G) =H^1(K,G)$. 
	Hence, after replacing $Y$ with a $K$-twist, we may assume that the class in $H^1({K((t^{\frac{1}{\infty}}))}, G)$ is trivial. Therefore it has to be trivial already in $H^1({K((t^{\frac{1}{n}}))}, G)$ for some $n \geq 1$.
	In other words, the $G$-torsor 
	$$
	\Spec K((t^{\frac{1}{n}})) \times_U V \to \Spec K((t^{\frac{1}{n}}))
	$$
	has a section.
	This section induces a commutative diagram as follows:
	\[
	\begin{tikzcd}
	& V \arrow[d] \\
	\Spec K((t^{\frac{1}{n}})) \arrow[r] \arrow[ru] & U                .
	\end{tikzcd}
	\]
	By the valuative criterion of properness (applied to $Y^{\sm} \to X$), we may extend the diagram above to the following:
	\[
	\begin{tikzcd}
	& Y^{\sm} \arrow[d] \\
	{\Spec K[[t^{\frac{1}{n}}]]} \arrow[r] \arrow[ru] & X                .
	\end{tikzcd} 
	\]
	Since the lower morphism specializes to $P$, the specialization of the diagonal morphism provides the sought lift of $P$.
\end{proof}

\section{Harari's question}\label{SSec:Reformulation}

\paragraph{Setup} Recall that $p:U \to \Spec K$ is a  geometrically integral smooth variety over a number field $K$, $X$ is a smooth compactification of $U$, and $\lambda:V \to U$ is a torsor under a finite group scheme $G/K$. We assume here that $G=A$ is commutative, and let $A'=\Hom(A,\G_{m,K})$ be its Cartier dual.  For every $v$, define $E_v:=\operatorname{im}\left(U\left(K_v\right) \rightarrow \mathrm{H}^1\left(K_v, A\right), P_v \mapsto[V|_{P_v}]\right)$ (this is not a subgroup in general).  Let $S$ be a finite set of places of $K$, let $\left(P_v\right)_{v \in S} \in \prod_{v \in S} U\left(K_v\right)$ , and $f_v \coloneqq [V|_{P_v}], \ v \in S$.

Let $\Br_{\lambda}U < \operatorname{Br}_1(U)$ be the subgroup generated by the cup-products $p^*b \cup[V]$, as $ b$ varies in  $\mathrm{H}^1(K, A')$, and let $B \coloneqq \Br_{\lambda}U \cap \operatorname{Br}(X)$. 

\begin{question}[Harari]\label{QHarari}
    Assume that there is no Brauer--Manin obstruction for $(P_v)_{v \in S}$ with respect to $B$. Does there exist then an $a \in \mathrm{H}^1(K, A)$ such that $a_v=f_v$ for all $v \in S$ and $a_v \in E_v$ for $v \notin S$?
\end{question}

Using Poitou-Tate duality, one may obtain a positive answer to Question \ref{QHarari} by replacing $E_v$ with the subgroup $\left\langle E_v\right\rangle$ of $\mathrm{H}^1(K_v, A)$ generated by it (we leave this as an exercise to the interested reader, or see \cite{Tesi}). However, as originally remarked by Harari, there is a big difference between $E_v$ and $\left\langle E_v\right\rangle$ in general!

The following proposition relates Question \ref{QHarari} with Question \ref{Q11}:

\begin{proposition}\label{Prop:reformulation}
    Assume that $\Br X/\Br_0 X$ is finite and $X(\A_K) \neq \emptyset$. Then the identity
    \[
    X(\A_K)^{\lambda}=X(\A_K)^{B }
    \]
    holds if and only if there exists a finite $S_0 \subseteq M_K$ such that, for all $S \supseteq S_0$, Question \ref{QHarari} has a positive answer.
\end{proposition}
\begin{proof}
    The assumption implies that $B$ is finite.

    We prove the forward implication first. Let $S_0$ be a set of places such that the Brauer--Manin pairing associated to $B$ is trivial outside $S_0$.
    Let $S \subseteq M_K$ be a finite set containing $S_0$, and $\left(P_v\right)_{v \in S} \in \left(\prod_{v \in S} U\left(K_v\right)\right)^B$. We wish to find an $a \in H^1(K,A)$ such that $a_v =f_v$ for all $v \in S$ and $a_v \in E_v$ for all $v \notin S$. Let $P_v, v \notin S$ be any point of $U(K_v)$. Note that $(P_v)_{M_K} \in X(\A_K)^{B}.$ Since $X(\A_K)^{\lambda}=X(\A_K)^{B},$ we may approximate arbitrarily well $(P_v)_{M_K}$ with an adelic point $(Q_{v})_{v \in M_K}$ such that there exists $a \in H^1(K,A)$ for which $(Q_{v})_{v \in M_K} \in \lambda_{a}(V_{a}(\A_K))$. In particular, $Q_v \in \lambda_{a}(V_{a}(K_v))$ for each $v$, or in other words the torsor $[V_a|_{Q_v}]$ over $K_v$ contains a $K_v$-point and is thus trivial. It follows that  $0=[V_a|_{Q_v}]=[V|_{Q_v}]-a_v\in H^1(K_v,A)$ for all $v$, and hence $a_v \in E_v$ for all $v$. Moreover, the map $[V|_{-}]:U(K_v) \to H^1(K_v,A)$ is locally constant, and thus  $[V|_{Q_v}]=[V|_{P_v}]=f_v$ for $v \in S$. The class $a$ is now the sought class.

    For the other direction, assume there is an $S_0$ as in the statement.  We need to show that $X(\A_K)^{\lambda}=X(\A_K)^{B }.$ We may assume after enlarging $S_0$ that its Brauer--Manin pairing on $X$ is trivial for all places $v \notin S_0$.

    Let $(P_v)_{M_K} \in X(\A_K)^{B}$. Since $B$ is finite, $X(\A_K)^{B} \subseteq X(\A_K)$ is open.  In particular, after an arbitrarily small approximation, we may assume that  $(P_v)_{M_K} \in U(K_{\Omega})^{B}$.    Then, for all $S \supseteq S_0$, our assumption that the answer to Question \ref{QHarari} is ``yes'' shows that there exists an $a \in H^1(K,A)$ such that $a_v=[V|_{P_v}]$ for $v \in S$ and $a_v = [V|_{Q_v}]$ for some $Q_v \in U(K_v)$ for $v \notin S$. Thus the $A$-torsor $V_a \to X$ specializes to the trivial torsor over $P_v, v \in S$ and over $Q_v, v \notin S$, and hence there exists a point $(R_v) \in V_a(K_{\Omega})$ whose image is $P_v$ for $v \in S$ and $Q_v$ for $v \notin S$. In particular, $V_a(K_{\Omega}) \neq \emptyset$ and therefore $V_a(\A_K) \neq \emptyset$, and we may thus modify $R_v$ so that $(R_v) \in V_a(\A_K)$ and so that the image of $R_v$ is $P_v$ for $v \in S$. We now modify $Q_v$ to $\lambda_a(R_v)$ for $v \notin S$, and by construction we have that $((P_v)_{v \in S}, (Q_v)_{v \notin S}) \in U(\A_K)^{\lambda}$.
    Enlarging $S$,  the points $((P_v)_{v \in S}, (Q_v)_{v \notin S})$ approximate arbitrarily well $(P_v)_{M_K}$, and thus we deduce that $(P_v)_{M_K} \in X(\A_K)^{\lambda}$. 
\end{proof}

% We conclude this section by showing how, as anticipated in the introduction, open descent theory \cite{opendescent} proves that the analog of Question \ref{Q11} has a positive answer when $M$ is a torus. Let $U$ be a geometrically connected $K$-variety, and $X$ a smooth compactification. Let $\lambda:V \to U$ be a torsor under a torus $M$. Let, as above, $b \in \mathrm{H}^1(k, M')$, $\Br_{\lambda}U:=\{p^*b \cup[V]\} \subseteq \operatorname{Br}_1(U)$, and $B=\Br_{\lambda}U \cap \operatorname{Br}(X)$.

% \begin{proposition}\label{Prop:Torus}
%     Assume that $\Br X/\Br_0 X$ is finite and that $M$ is a torus. Then $X(\A_K)^{\lambda}=X(\A_K)^{\Br X \cap \Br_{\lambda} U }.$
% \end{proposition}
% \begin{proof}
%     Since $M$ is a torus, $M'$ is finitely generated and torsion-free, and thus $\mathrm{H}^1(k, M')$ is finite and so is $\Br_{\lambda}U$. Hence $U(\A_K)^{\Br_{\lambda}U}$ is dense in $X(\A_K)^{\Br X \cap \Br_{\lambda} U }$ by Harari's formal lemma. The former is equal to $U(\A_K)^{\lambda}$ by open descent theory \cite{opendescent}, and we infer that the adelic closure $X(\A_K)^{\lambda}$ of $U(\A_K)^{\lambda}$ in $X(\A_K)$ is equal to $X(\A_K)^{\Br X \cap \Br_{\lambda} U }$, as wished.
% \end{proof}

\section{A Brauer--Manin obstruction to ramified descent}\label{Sec:Obstruction}

We recall that $\lambda:V \to U$ is a torsor under a finite group scheme $G/K$, that $U$ is smooth and geometrically integral over $K$, that $X$ is a compactification of $U$, and that $Y \to X$ is the relative normalization of $X$ in $V$.
We defined:
\[
X(\A_K)^{\lambda}=\overline{{\bigcup_{\xi \in H ^1(K,G)}\lambda_{\xi}(V_{\xi}(\A_K))} }
\]

We define in this section a subgroup $\Br_{\lambda}^{ram}X \subseteq \Br X$  such that $X(\A_K)^{\lambda} \subseteq X(\A_K)^{\Br_{\lambda}^{ram}X}$. The group $\Br_{\lambda}^{ram}X$ may be transcendental, as we show in Section \ref{Sec:GBsp}.

\subsection{Definition of $\Br_{\lambda}^{ram}X$}

Let $\mathfrak{l}$ be the composition $V_{\oK} \to  V \to[\lambda] U$. There are natural $G(\oK)$- and $\Gamma_K$-actions on $V_{\oK}$, the first induced by the $G$-action on $V$, and the second via the second factor of $V_{\oK}=V \times_K \Spec \oK$.

\begin{lemma}\label{Lemprofinite}
    The $G(\oK)$- and the $\Gamma_K$- actions generate a $(G(\overline{K}) \rtimes \Gamma_K)$-action on $V_{\oK}$. The profinite \'etale cover $\mathfrak{l}:V_{\oK} \to U$ is a profinite torsor under this $(G(\overline{K}) \rtimes \Gamma_K)$-action.
\end{lemma}
\begin{proof}
    Let $L/K$ be a field extension. The points of $V_L$ with values in a $K$-algebra $R$ are described by
    \begin{equation}\label{EqDisjUnion}
        V_L(R)=\bigsqcup_{\iota:L \hookrightarrow R} V(R),
    \end{equation}
    where $\iota$ ranges among all $K$-embeddings of $L$ in $R$. Consider the natural actions:
    \begin{enumerate}[label=(\roman*)]
        \item the right action of $G(L)$ on $V_L$ defined by letting $G(L)$ act via the map $G(L) \to G(R)$ on each disjoint set appearing in \eqref{EqDisjUnion};
        \item when $L/K$ is Galois, the right $\Gal(L/K)$-action defined by letting $\gamma \in \Gal(L/K)$ act via $\iota \mapsto \iota \circ \gamma^{-1}$.
    \end{enumerate}

    For $g \in G(L)$ and $\gamma \in \Gal(L/K)$, we have $\leftidx^{\gamma}g \cdot \gamma \cdot x= \gamma \cdot g \cdot x$, where $x$ is an $R$-point of $V_L$. By this relation, the two actions above generate a $(G(L) \rtimes \Gal(L/K))$-action on $V_L$. This action is fixed-point-free and it commutes with the projection $V_L \to U$. In addition, when $L/K$ is finite and splits $G$ (i.e.\ when $G(L)=G(\oK)$), letting $\Gamma \coloneqq G(L) \rtimes \Gal(L/K)$, we have $|\Gamma| = |G(\oK)|\cdot [L:K] = \deg (V_L \to U)$. Thus the morphism $(\mu,id):\Gamma \times_U V_L \to V_L \times_U V_L$, where $\mu$ denotes the $\Gamma$-action on $V_L$, is an injective morphism of finite \'etale covers of $U$ of the same degree, and hencean isomorphism. In other words, $V_L \to U$ is a torsor under $\Gamma$.
    
    When $L=\oK$, the two actions (i) and (ii) are the ones described before the statement of the lemma, and taking an inverse limit over all finite Galois subextensions $L \subset \oK$ that split $G$ (these form a cofinal subset) finishes the proof.
\end{proof}

Let $\Gamma_G \coloneqq G(\overline{K}) \rtimes \Gamma_K$. Through the construction of Section \ref{Sec:Notation}, the profinite \'etale torsor $V_{\oK} \to U$ under $\Gamma_G$ gives rise to a \v{C}ech-to-\'etale map on cohomology:
\begin{equation}\label{Eq:Def}
\check{C}_{\mathfrak{l}}:{H}^2(\Gamma_G,\oK[V]^*) \to {H}^2(U,\G_{m}),
\end{equation}
where $\Gamma_G$ acts on $\G_{m}(V_{\oK})= \oK[V]^*$ by pullback. The restriction of this $\Gamma_G$-action on $\oK^* \subseteq \oK[V]^*$ is equal to the pullback of the natural $\Gamma_K$-action along the projection $\Gamma_G \to \Gamma_K$. Hence we have a natural morphism:
\[
H^2({\Gamma_G, \oK^*}) \to {H}^2(\Gamma_G,\oK[V]^*)={H}^2(\Gamma_G,\G_{m}(V_{\oK})),
\]
where the implied action on the LHS is the pullback described above.

\begin{definition}\label{Def:TheObstruction}
	We define the subgroup $\Br_{\lambda}^{ram}(U)$ of $\Br U$ as the image of the composition 
	$$ H^2({\Gamma_G, \oK^*}) \to {H}^2(\Gamma_G,\oK[V]^*)\to[\check{C}_{\mathfrak{l}}] {H}^2(U,\G_{m})=\Br U.$$ 
	We define $\Br_{\lambda}^{ram}X \subseteq \Br(X)$ as the intersection $\Br(X) \cap \Br_{\lambda}^{ram}(U)$.
\end{definition}

\begin{remark}
    The fact that $U$ does not appear in the notation ``$\Br_{\lambda}^{ram}X$'' is justified by the fact that $\Br_{\lambda}^{ram}X$ may be defined purely in terms of the ramified $G$-cover $Y \to X$ (and, in fact, only in terms of the ``generic fiber'' $G$-torsor $\Spec K(Y) \to \Spec K(X)$). Indeed
    $
    \Br_{\lambda}^{ram}X =\Br(X) \cap \Br_{\lambda}^{ram}(K(X)),$ 
    where $\Br_{\lambda}^{ram}(K(X)) \subseteq \Br (K(X))$
    is defined as the image of $H^2(\Gamma_G,\oK^*)$ in $H^2(K(X),\G_m)$ through the morphism:
    \[
    H^2(\Gamma_G,\oK^*) \to H^2(\Gamma_G,\oK(Y)^*) \to[\check{C}_{\mathfrak{l}|_{K(X)}}] H^2(K(X),\G_m).
    \] 
\end{remark}

\begin{remark}\label{Rmk:Piceq0}
        When $\oK[V]^*/\oK^*=\Pic V_{\oK}=0$, the Hochshild-Serre spectral sequence $H^i(\Gamma_G,H^j(V_{\oK},\G_m)) \Rightarrow H^{i+j}(U, \G_m)$ yields the short exact sequence $0 \to H^2(\Gamma_G,\oK^*) \to \Br U \to \Br V_{\oK}.$
	Hence, in this case, $\Br_{\lambda}^{ram}(U)=\Ker (\Br U \to \Br V_{\oK})$ and
	\[
	\Br_{\lambda}^{ram}X=\Br_{\lambda}^{ram}(U) \cap \Br X=\Ker (\Br X \to \Br V_{\oK}).
	\]
\end{remark}

Although the condition $\oK[V]^*/\oK^*=\Pic V_{\oK}=0$ is rarely satisfied in practice, we give some examples below.
\begin{example}\label{ExTutto0}
    (I). If $V$ is a simply connected semi-simple algebraic group, then  $\oK[V]^*/\oK^*=\Pic V_{\oK}=0$. To get an example of a $G$-torsor $V \to U$, we may take as $G$ any finite subgroup(-scheme) of $V$, and define $U \coloneqq V/G$.

    (II). If $V$ is a universal torsor of a smooth proper rationally connected variety (as defined in \cite{CTS87}), one has  $\oK[V]^*/\oK^*=\Pic V_{\oK}=0$ (see (2.1.1) of {\em loc.cit.}). We then let $G$ be any finite subgroup of the N\'eron-Severi torus $T$ (i.e.\ the one under which $V$ is a torsor), and let $U =V/G$.
 
    (III). Finally, we give an example where both $V$ and $U$ are open K3 surfaces. Let $\mathcal E \to \P^1_K$ be an elliptic K3 surface with Mordell-Weil group $\Z/2\Z$, with (exactly) two reducible fibers of Kodaira types $I_{10}^*$ and $I_2$, and such that the $0$-section $O$ and the $2$-torsion section $\tau$ intersect different components of the $I_2$-fiber. Such an $\mathcal E$ exists, and it may be defined over $\Q$, see e.g.\ 5.2 and 5.3 of \cite{Kumar}, where $\mathcal E$ is realized as the double-cover of the Kummer surface associated to the Jacobian of a curve of genus $2$. Then the Picard group of $\mathcal E_{\oK}$ is of rank $16$, freely generated by the fourteen components of the $I_{10}^*$-fiber, by $O$ and by $\tau$. In particular, letting $V$ be the complement of these divisors in $\mathcal E$, we have that $\oK[V]^*=\oK^*$ and $\Pic \bar V=0$. The involution on $\mathcal E$ induced by $\tau$ restricts to an involution of $V$, let $G\cong \Z/2\Z$ be the group generated by it. A smooth minimal compactification $X$ of the quotient $U=V/G$ is an elliptic surface isogenous to the original K3 elliptic surface $\mathcal E \to \P^1$, and is thus K3. (In fact, if $\mathcal E$ is as in \cite{Kumar}, the quotient $U=V/G$ is a Kummer surface, see {\em loc.cit.}).
\end{example}

\begin{proposition}\label{Prop:InclusionBraBr}
    Assume that $G=A$ is commutative. Then $ \Br(X) \cap \Br_{\lambda}U$ is contained in $\Br_{\lambda}^{ram}X $.
\end{proposition}

\begin{proof}
    We prove the stronger inclusion $ \Br_{\lambda}U \subseteq \Br_{\lambda}^{ram}U.$ Recall that $\Br_{\lambda}U$ is generated by the cup-products $p^*b \cup [V]$, as $b$ varies in $H^1(K,A')$. Fix a $b \in H^1(K,A')$. Both classes $p^*b$ and $[V]$ trivialize when pulled back along the pro-\'etale cover $V_{\overline{K}} \to U$. Hence, following e.g.\ \cite[p.18]{Skorobogatov}, there exist  \v{C}ech-cocycles $b_V\in \check{H}^1(V_{\oK}/U, A'),\alpha_V \in \check{H}^1(V_{\oK}/U, A)$ that represent $p^*b$ and $[V]$, i.e.\ such that $\check{C}_{\mathfrak{l}}(b_V)=p^*b$ and $\check{C}_{\mathfrak{l}}(\alpha_V)=[V]$. 

    We have a commutative diagram of cup-products \cite[Corollary 3.10]{Swan}
    \[
    % https://tikzcd.yichuanshen.de/#N4Igdg9gJgpgziAXAbVABwnAlgFyxMJZABgBpiBdUkANwEMAbAVxiRAAIAdTgYwAsYPANbAAEgF8AegEYAFADUA+sG4QA0uID0AVVLsAggHIAlOxDjS6TLnyEUAJnJVajFmy68BwsVLlKVnOpaugam5pYgGNh4BETSTtT0zKyIINx4ALbw4VbRtkQAzAkuyWzc-IIiEpL2CsqqGjqkDZIAVMY5kdYxdshF0s5JbqnlXlVStf4Nwc2BasjyFG0dFrk2sQ6kA4muKRyjlT4ydQFBTaFmq115GyRbg7vuB97VfvVzMwYmlxFR673xbYlYZpTiZbJXP49IhkewPUqpUTHXRGFa-br5TZwnYIkBIuQotFraEoIrY4F7JG1XTcADiigyROu-zipHJQz26SwWQQ4mcMCgAHNsihQAAzABOEAySDIIBwECQADZqAI6FA2JAwKwcSCALydSXSpDxeWKxAAVlVMHVmoIOopbANVyNMsQRTNSAALC6pW6VZ7EAAOXWczyHADC4neGToOD4kroQnYDHEhr9SCtgYAnKGyuHvFGY3GExKkym077jYgvdQFSa5Rz8xVC9GArH44nk6n09Xc4HpMQq26s-Wa3mRrwmGhe27TWOPU3Jzxp+YKOIgA
    \begin{tikzcd}
    { \check{H}^1(V_{\oK}/U, A') } \arrow[d, "=", no head]              & \times & { \check{H}^1(V_{\oK}/U, A) } \arrow[d, "=", no head] \arrow[r, "\cup"]              & {\check{H}^2(V_{\oK}/U,\oK^*)} \arrow[d]                               \\
    { \check{H}^1(V_{\oK}/U, A') } \arrow[d, "\check{C}_{\mathfrak l}"] & \times & { \check{H}^1(V_{\oK}/U, A) } \arrow[d, "\check{C}_{\mathfrak l}"] \arrow[r, "\cup"] & {\check{H}^2(V_{\oK}/U,\oK[V]^*)} \arrow[d, "\check{C}_{\mathfrak l}"] \\
    {H^1(U,A')}                                                         & \times & {H^1(U,A)} \arrow[r]                                                                 & {H^2(U,\G_m)},
    \end{tikzcd}
    \]
    see e.g.\ \cite{Swan} for the definition of the \v{C}ech cup-product on the first two rows.
    In particular, we have that $\check{C}_{\mathfrak{l}}(b_V \cup \alpha_V)=\check{C}_{\mathfrak{l}}(b_V)  \cup \check{C}_{\mathfrak{l}}(\alpha_V)=p^*b \cup [V]$, and $p^*b \cup [V]$ belongs to the image $\Br_{\lambda}^{ram}U$ of the composition $\check{H}^2(V_{\oK}/U,\oK^*) \to \check{H}^2(V_{\oK}/U,\oK[V]^*) \to[\check{C}_{\mathfrak l}] H^2(U,\G_m)$. 
\end{proof}

We go on to prove Proposition \ref{Prop2}, i.e.\ that $X(\A_K)^{\lambda}$ is contained in $X(\A_K)^{\Br_{\lambda}^{ram}X}$. 
I profoundly thank one of the anonymous referees, who provided the following proof, which simplifies a lot the previous one the author had in mind.

\begin{proof}[Proof of Proposition \ref{Prop2}]
    We shall prove that $U(\A_K)^{\lambda} \subseteq U(\A_K)^{\Br_{\lambda}^{ram}U}$, from which the statement follows by taking adelic closures in $X(\A_K)$ and noting that $U(\A_K)^{\Br_{\lambda}^{ram}U} \subseteq X(\A_K)^{\Br_{\lambda}^{ram}X}$.

    Let $(P_v)_{v \in M_K} \in U(\A_K)^{\lambda}$ and let $\xi \in Z^1(K,G)$ be such that $P_v= \lambda_{\xi}(Q_v), {v \in M_K}$ for some $(Q_v)_{v \in M_K} \in V_{\xi}(\A_K)$. The cocycle $\xi$ defines a section $s_{\xi}=(\xi,id)$ of the surjection:
    \[
    \Gamma_G = G(\oK) \rtimes \Gamma_K \to  \Gamma_K \to 1.
    \]
    Recall that $V_{\oK}$ is naturally endowed with a $\Gamma_G$-action that makes $V_{\oK} \to U$ a profinite $\Gamma_G$-torsor. One easily checks from the definition of twist \cite[p.12]{Skorobogatov} that
    \[
    V_{\xi}= V_{\oK} / s_{\xi}(\Gamma_K).
    \]
    Thus we have a commutative diagram:
    \[
    % https://tikzcd.yichuanshen.de/#N4Igdg9gJgpgziAXAbVABwnAlgFyxMJZABgBpiBdUkANwEMAbAVxiRAB12IaYAnBrGBjAAagF8QY0uky58hFAEZyVWoxZtO3PgKGiJUmdjwEiZRavrNWiEAFVJ0kBmPyiyi9SsbbIgPrAnAAeWAaqMFAA5vBEoABmvBAAtkhkIDgQSABMXuo2HOwA4nRJSXR+hY7xiSmIAMzUGdmGIAnJSMrpmYhp3vkAvFWtNR2N3Q1q1ppFJWV+ANKSFGJAA
\begin{tikzcd}
\overline{V} \arrow[d, "\Gamma_G"] & \overline{V} \arrow[l, "="] \arrow[d, "\Gamma_K"] \\
U                                  & V_{\xi} \arrow[l]                                
\end{tikzcd},
    \]
    equivariant with respect to $s_{\xi}:\Gamma_K \to \Gamma_G$,
    which induces by functoriality of $\check{C}$ the following commutative diagram:
    \[
    % https://tikzcd.yichuanshen.de/#N4Igdg9gJgpgziAXAbVABwnAlgFyxMJZABgBpiBdUkANwEMAbAVxiRAAkA9AJgAoAdfgHE6AW1F0A+kNIACQRADSnAFQBKEAF9S6TLnyEUZAIxVajFmy59BI8VMVyFy9Vp0gM2PASLHSp6npmVkQQQQBjAAsYcIBrYHZNHgF+CBoYACcGLDAYYAA1TQB6fMlgQQAPLG1bSVENbV0vA19yMyDLUIjouISkm1T0rJy8wqKAVVJa+rcm-R8UbjbAixCw-gAhDNlx2Y89b0NkJYDzYLZBLdlS8v4qzS0zGCgAc3giUAAzDIhRJDIQDgIEg-GdOiAMvBOLc7BJpJoynAypVqilYQ41A9GiBvr8kABmahApBLMFrQQMMQAIygUlu2BeEiSKj2uL+iABxMQ+OxbJBROBiG4vJ+7MJgMFABYVucuvwojF4gBhBG3Kl0DK3SmiGl0TRY9x8oUCpAAVhl4O6iuAKuRg0y2VyBWKNxR+tZoqQ0olZot5P42t1doZTNUj00QA
\begin{tikzcd}
{H^2(\Gamma_G, \oK^*)} \arrow[d, "res^{\Gamma_G}_{s_{\xi}(\Gamma_K)}"] \arrow[r] & {\check{H}^2(\overline{V}/U,\G_m)} \arrow[d, "\lambda_{\xi}^*"] \arrow[r, "\check{C}_{\mathfrak{l}}"] & \Br U \arrow[d, "\lambda_{\xi}^*"] \\
\Br K = {H^2(\Gamma_K, \oK^*)} \arrow[r]                                                 & {\check{H}^2(\overline{V}/V_{\xi},\G_m)} \arrow[r, "\check{C}_{\overline{V}/V_{\xi}}"]                    & \Br V_{\xi}                          
\end{tikzcd}.
    \]
    Recalling that $\Br_{\lambda}^{ram}U$ is defined to be the image of the upper composition, we deduce by the commutativity that   $\lambda_{\xi}^*\Br_{\lambda}^{ram}U \subseteq \Br_0 V_{\xi}$. Thus, for any $B \in \Br_{\lambda}^{ram}U$:
    \[
    ((P_v),B)_{BM}=((Q_v),\lambda_{\xi}^*B)_{BM}=0,
    \]
    as wished.
\end{proof}

In summary, we have the following series of inclusions:
\begin{equation}\label{EqInclusions}
    X(\A_K)^{\lambda}\subseteq X(\A_K)^{\Br_{\lambda}^{ram}X} \subseteq X(\A_K)^{\Br X \cap \Br_{\lambda}U}.
\end{equation}

However, in contrast to what happens on $U$ (where the inclusions \linebreak $U(\A_K)^{\lambda}\subseteq U(\A_K)^{\Br_{\lambda}^{ram}U} \subseteq U(\A_K)^{\Br_{\lambda}(U)}$ are actually identities by \cite{opendescent}), the last inclusion in \eqref{EqInclusions} may well be strict! (See Section \ref{Sec:GBsp}).

\section{Unramified Brauer groups of $SL_n/H$, $H$ metabelian}\label{Sec6}

For Section \ref{Sec6} fix a finite group scheme $H$ over a field $k$ of characteristic $0$. Let $B=[H,H]$, and $A=H/[H,H]$. We assume that $H$ is {\bf metabelian}, i.e.\ that $B$ is commutative.

Let $U=SL_{n,k}/H$, and $V=SL_{n,k}/B$. As $B$ is normal in $H$ with quotient $A$, the morphism $\lambda:V \to U$ is an $A$-torsor, where the $A$-action on $V$ is the one that on $\ok$-points is given by
$
(SL_{n,K}/B) \times A \to SL_{n,K}/B,
(xB,a) \mapsto xBa=xaB.
$ 

Recall that $\mathfrak{l}:\overline{V} \to U$ induces a \v{C}ech-to-\'etale map on cohomology:
\[
\check{C}_{\mathfrak{l}}:H^2(\Gamma_A, \oK^*) =H^2(\Gamma_A, \oK[V]^*) \to \Br U.
\]

 The author thanks Olivier Wittenberg for making him notice the following:

\begin{theorem}\label{Thm:BrSLnG}
    If $\Sha^1_{\omega}(K,B')=0$, then $\Br X = \Br_{\lambda}^{ram}X$.
\end{theorem}

(Recall that $X$ denotes a smooth compactification of $U$.)

\begin{proof}
It suffices to prove that $\Br_{ur} U \subseteq \im H^2(\Gamma_A, \overline{K}^*)$.
	Consider the Hochschild-Serre spectral sequence of $\overline{V} \to[\Gamma_G] U$, keeping in mind that $\overline{K}[V]^*=\overline{K}^*$ and $\Pic \overline{V} =B'$, we get the following exact sequence:
	\begin{equation}\label{Eq:HSSS610}
	H^2(\Gamma_A,\overline{K}^*) \to \Ker (\Br U \to \Br \overline{V}) \to H^1(\Gamma_A,B')
	\end{equation}
	Since $\Sha^1_{\omega}(K,B')=0$, we have $\Br_{ur} V=\Br K$ by \cite[Prop.\  4]{Harari2007}. Thus any element of $\Br_{ur} U$ lies in the kernel of $\Br U \to \Br \overline{V}$. In addition, we know that after base-changing to $\overline{K}$ every element of  $\Br_{ur} U$ comes from $H^2(A,\overline{K}^*)$ (see \cite[Lemma 5.1]{BogomolovMumbai}). In particular, $\Br_{ur} U$ maps to $\Ker(H^1(\Gamma_A,B') \to[\res] H^1(A,B'))$ in the sequence above. By the inflation-restriction five-term sequence of $A \trianglelefteq \Gamma_A$, we have that $\Ker(H^1(\Gamma_A,B') \to H^1(A,B'))=H^1(\Gamma_K,B')$. For $\beta \in \Br_{ur} U$, we denote by $\delta(\beta)$ its image in $H^1(\Gamma_K,B')$.

	Finally, by functoriality of the Hochschild-Serre spectral sequence, we get the following commutative diagram:
	\[
	\begin{tikzcd}
	{H^2(\Gamma_A,\overline{K}^*)} \arrow[d] \arrow[r] &  \Ker (\Br U \to \Br \overline{V}) \arrow[d] \arrow[r] & {H^1(\Gamma_A,B')} \arrow[d] \\
	{H^2(\Gamma_K,\overline{K}^*)} \arrow[r]           &  \Ker (\Br V \to \Br \overline{V}) \arrow[r]           & {H^1(\Gamma_K,B')}, 
	\end{tikzcd}
	\]
	where the second row is just \eqref{Eq:HSSS610} with $A=0$. Hence (again because $\Br_{ur}V=\Br K$) every element of $\Br_{ur} U$ has to map to $0$ in $ {H^1(\Gamma_K,B')}$. This implies that $\delta(\beta)=0$ for every $\beta$. Hence $\Br_{ur}U \subseteq \im H^2(\Gamma_A,\overline{K}^*) =\Br_{\lambda}^{ram}U$, therefore $\Br X=\Br_{ur}U=\Br_{\lambda}^{ram}U \cap \Br X= \Br_{\lambda}^{ram}X$, as wished.
\end{proof}

By Chebotarev's theorem, $\Sha^1_{\omega}(K,B')=0$ when $B'$ is constant. We thus get:

\begin{corollary}
    If $B$ is constant of exponent $e$ and $\mu_e \subseteq K^*$, then $\Br X = \Br_{\lambda}^{ram}X$.
\end{corollary} 

Proposition \ref{Prop:gbsp} is a consequence of the above corollary.

\subsection{Nilpotent $H$}\label{Sec:GBsp}

For the rest of the section we assume that $B$ is central in $H$ and that $H$ is constant of prime exponent $p$, where $p \neq 2$ and $\mu_p \subseteq k^*$. We choose a primitive $p$-th root of unity, and use it to identify throughout this subsection $\mu_p$ with $\Z/p\Z$.

In particular, we shall always tacitly identify $\mu_p$ with $\Z/p\Z$, after making the implicit choice of a $p$-th root of unity. 

Our aim in this subsection is to explicitly describe $\Br_{ur}U$, see Theorem \ref{ThmBogomolovext} below.
We may naturally associate to the central extension 
\[
1 \to B \to H \to A \to 1
\]
a class $[H] \in H^2(A,B)$ (see \cite[Sec.\  IV.3]{Brown}). We denote by $[-,-]: A \times A \to B$ the map that sends $a_1,a_2 \in A$ to the commutator of (any two) lifts $\bar{a}_1,\bar{a}_2$ of them in $H$.
Let $\mathfrak{c}: \Lambda^2A \to B$ be the natural homomorphism $a_1\wedge a_2 \mapsto [a_1,a_2]$. 

Recall that $A^D=\Hom(A, \qz)$. We shall make frequent use of the identification:
\[
\Lambda^2(A^D) = (\Lambda^2 A)^D, \alpha_1 \wedge \alpha_2 \mapsto \left(a_1\wedge a_2 \mapsto \alpha_1(a_1) \alpha_2(a_2)- \alpha_1(a_2) \alpha_2(a_1)\right).
\]
We denote $\Lambda^2(A^D) = (\Lambda^2 A)^D$ with $\Lambda^2A^D$.
Let $(\Lambda^2 A^D)_{bic} \subseteq \Lambda^2 A^D$ be the subgroup of elements $\beta \in \Lambda^2 A^D$ such that $\beta(a_1 \wedge a_2)=0$ for any $a_1,a_2 \in A$ with $[a_1,a_2]=0$.

% \begin{lemma}\label{Lem}
%     The image of $\mathfrak{c}^D:B^D \to (\Lambda^2 A^D)$ is contained in $(\Lambda^2 A^D)_{bic}$ and we have a natural exact sequence:
%     \[
%     B^D \to[\mathfrak{c}^D] (\Lambda^2 A^D)_{bic} \to S^D \to S_{\Lambda}^D \to 1
%     \]
% \end{lemma}
% \begin{proof}
%     Since $S = \Ker (\Lambda^2A \to B)$, we have $S^D= \Coker(B^D \to \Lambda^2A^D)$. For an element $\alpha \in \Lambda^2A^D$, its projection in $S^D$ maps to $0$ in $S_{\Lambda}^D$ if and only if $\alpha(a_1\wedge a_2)=0$ for any $a_1, a_2 \in A$ with $[a_1,a_2]=0$, i.e.\ if and only if $\alpha \in (\Lambda^2 A^D)_{bic}$. This provides an identification between $\Ker(S^D \to S_{\Lambda}^D)$ and $\Coker(B^D \to (\Lambda^2 A^D)_{bic})$ and proves both parts of the lemma.
% \end{proof}

Let $\xi_U:\Lambda^2A^D=\Lambda^2H^1(A, \Z/p\Z) \to \Br U$ be the composition of the three maps:
\begin{align}
    \Lambda^2 H^1(A, \Z/p\Z) \to[\cup] H^2(A, \Z/p\Z) & & \text{(cup-product)} \\
    H^2(A, \Z/p\Z)=H^2(A, \mu_p) \to[\check{C}_{\lambda}] H^2(U,\mu_p) &  &\text{(\v{C}ech-to-\'etale map)} \\
    H^2(U,\mu_p) \to H^2(U,\G_m) = \Br U &  &\text{(changing coefficients)}
\end{align}
We also define analogously a map  $\xi_{\overline{U}}:\Lambda^2 A^D \to  \Br \overline{U}.$  
The following is a reformulation of a result of Bogomolov (see Lemma 5.1 of \cite{BogomolovMumbai}):

\begin{theorem}[Bogomolov, reformulated]\label{ThmBogomolov}
    Assume that $k$ is algebraically closed. The image of $(\Lambda^2 A^D)_{bic}$ under $\xi_U$ is unramified, and the following sequence is exact:
    \begin{equation}
        \label{Bogourkclosed}  B^D \to[\mathfrak{c}^D] (\Lambda^2 A^D)_{bic} \to[\xi_U] \Br_{ur} U \to 1.
    \end{equation} 
\end{theorem}

We introduce the following notation used in the proof. Let \(\gamma \in \Lambda^2 A^D = \linebreak \operatorname{Hom}(\Lambda^2 A, \mathbb{Q} / \mathbb{Z})\) be a cocycle. Define \(H_{\gamma}\) as the central extension of \(A\) by \(\mathbb{Q} / \mathbb{Z}\), characterized by the following property: for any \(a_1, a_2 \in A\) the commutator of their corresponding lifts \(\bar{a}_1, \bar{a}_2\) in \(H_{\gamma}\) is
\[
\bar{a}_1 \bar{a}_2 \bar{a}_1^{-1} \bar{a}_2^{-1} = \gamma(a_1 \wedge a_2) \in \mathbb{Q} / \mathbb{Z}.
\]

\begin{proof}[Proof of Theorem \ref{ThmBogomolov}]
    By Lemma 5.1 of \cite{BogomolovMumbai}, $\Br_{ur} U$ is isomorphic to $\Coker( B^D \to (\Lambda^2 A^D)_{bic})$ (the group $S/S_{\Lambda}$ appearing in {\em loc.cit.} is the dual of $\Coker(B^D \to (\Lambda^2 A^D)_{bic})$). It only remains to show that the induced homomorphism $\xi_{\text{Bog}}:(\Lambda^2 A^D)_{bic} \to \Br_{ur} U$ is $\xi_U$. This follows from the proof of \cite[Lemma 5.1]{BogomolovMumbai}, but we include a detailed proof for completeness.
     
     One may read from pp.\ 462 and 469 in \cite{BogomolovMumbai} that the induced morphism $\xi_{\text{Bog}}:(\Lambda^2 A^D)_{bic} \to \Br_{ur}U$  is the restriction of the composition :
    \begin{equation}\label{Bogomolovmap}
       \Lambda^2 A^D \cong H^2(A, \qz) = H^2(A, \mu_{\infty}) \to H^2(A, k^*) \to[\check{C}_{\lambda}] \Br U, 
    \end{equation}
    where the last map is defined via the Hochschild-Serre spectral sequence, and the first isomorphism is defined by sending an element $\gamma \in \Lambda^2 A^D =\Hom(\Lambda^2A, \qz)$ to the class in $H^2(A, \qz)$ of the central extension $H_{\gamma}$ of $A$.
    %(See the discussion in p.\ 462 {\em loc.cit.}, and the isomorphism ``$\Lambda^2 (H^{ab})^D \cong H^2(H^{ab},\qz)$'' in p.\ 469 {\em loc.cit.}.) 
    The following lemma then shows that this composition is $\xi_U$.
    \end{proof}

    \begin{lemma}
        The isomorphism $\Lambda^2 A^D \cong H^2(A,\qz), \gamma \mapsto [H_{\gamma}]$ coincides with the composition $\Lambda^2 A^D \to[\cup] H^2(A,\Z/p\Z) \to H^2(A,\qz)$.
    \end{lemma}
    \begin{proof}
        For a monomial $\gamma = \gamma_1 \wedge \gamma_2 \in \Lambda^2 A^D$, the extension $H_{\gamma}$ may be realized as the set $\qz \times A$ with multiplication defined by $(b_1,a_1)\cdot (b_2,a_2)= (b_1+b_2+\gamma_1(a_1) \gamma_2(a_2),a_1+a_2)$. The class $[H_{\gamma}] \in H^2(A,\qz)$ is then represented  by the $2$-cocycle $a_1,a_2 \mapsto \gamma_1(a_1)\gamma_2(a_2)$ \cite[p.92]{Brown}. This is precisely the cup product $\gamma_1 \cup \gamma_2$ \cite[Proposition 1.4.8]{GermanBook}. Thus the two maps coincide on monomials  $\gamma_1 \wedge \gamma_2$. Since these span $\Lambda^2 A^D$, the two maps coincide, as wished.
    \end{proof}

%    It follows from the lemma that $\text{Bog}$omolov's map is the composition of the cup-product $\Lambda^2A^D \to H^2(A,\mu_p)$, the ``$\mu_p$ to $\G_m$'' coefficient enlargement map $H^2(A,\mu_p) \to H^2(A,k^*)=H^2(A,\G_m(V))$, and the \v{C}ech-to-\'etale map $\check{C}_{\lambda}: H^2(A,k^*) \to H^2(U, \G_m)$. By the naturality of $\check{C}_{\lambda}$, the last two operations (enlraging the coefficients and mapping \v{C}ech cohomology to \'etale cohomology) commute, and thus the composition \eqref{Bogomolovmap} is $\xi_U$, finishing the proof.

Let $\Br_e U = \Ker (e^* : \Br U \to \Br K)$, where $e$ is the $K$-point corresponding to the equivalence class of the identity in $SL_n/H$. Recall that we have a direct sum decomposition $\Br U = \Br_e U \oplus \Br_0 U$, where $\Br_0 U \cong \Br K$ denotes the constant Brauer elements. Let $\Br_{e,ur}U =\Br_e U \cap \Br_{ur}U.$ We now remove the algebraically closed assumption from Theorem \ref{ThmBogomolov}:

\begin{theorem}\label{ThmBogomolovext}
    The image of $(\Lambda^2 A^D)_{bic}$ under $\xi_U$ is contained in $\Br_{e,ur}U$, and the following sequence is exact:
    \begin{equation}
       % \label{Bogo}  B^D \to[\lambda^D] &\Lambda^2 A^D \to[\xi_U] \Br_e U,\\
        \label{Bogour}  B^D \to[\mathfrak{c}^D] (\Lambda^2 A^D)_{bic} \to[\xi_U] \Br_{e,ur} U \to 1.
    \end{equation} 
\end{theorem}
\begin{proof}
    The $A$-torsor $\lambda$ restricts to the trivial torsor over $e \in U(K)$, which implies that the image of $\xi_U$ is contained in $\Br_eU$ (see Remark \ref{Rmk:trivialtorsorCech}).

    By Bogomolov's theorem \ref{ThmBogomolov} and the key lemma \ref{keylemma} below, we have that $\xi_U((\Lambda^2A^D)_{bic}) \subseteq \Br_{ur}U$.     
    Finally, the composition
    \[
    \xi_{\overline{U}}:(\Lambda^2A^D)_{bic} \to[\xi_U] \Br_{e,ur} U \hookrightarrow \Br_{ur} \overline{U}.
    \]
    is surjective with kernel $\im (\mathfrak{c}^D:B^D \to \Lambda^2 A^D)$ by Bogomolov's Theorem, while the second morphism is injective by \cite[Proposition 5.9]{Arteche}. Hence the last map is an isomorphism, and the statement follows.
\end{proof}

    \begin{lemma}\label{keylemma}
    For $\alpha \in \Lambda^2A^D$, $\xi_U(\alpha) \in \Br_e U$ is unramified if and only if   $\xi_{\overline{U}}(\alpha) \in   \Br \overline{U}$ is unramified.
    \end{lemma}
    
    \begin{proof}    
    The forward implication is clear, we prove the converse. Let $\alpha \in \Lambda^2 A^D$ be an element whose image in $\Br \overline{U}$ is unramified. By Bogomolov's Theorem \ref{ThmBogomolov}, this unramifiedness is equivalent to $\alpha$ lying in $(\Lambda^2 A^D)_{bic}$.

    We first assume that $k=K$ is a number field. Let $v$ be a finite place of $k$ coprime with $p$. Let $P \in U(K_v)$. Let $\psi$ be the projection $SL_n \to SL_n/H$. Choose a geometric point $\bar P \in \psi^{-1}(P)(\overline{K_v})$, and let $f:\Gamma_v \to H$ be the homomorphism defined by $\gamma \cdot \bar P = \bar P \cdot f(\gamma)$. Since $v\nmid p$ the homomorphism $f$ factors though the tame decomposition group $\Gamma_v^{tame}$, which is topologically generated by elements $\iota$ and $\phi$ satisfying $\phi\iota\phi^{-1}=\iota^{Nv}$. Since $\mu_p \subseteq K^*$, we have $Nv \equiv 1 \bmod p$. In particular, since the exponent of $H$ is $p$, the images of $\iota$ and $\phi$ in $H$ commute. The group $A'=\langle f(\iota), f(\phi) \rangle$ is then bicyclic. The point $P$ lifts via $(SL_n/A')(K_v) \to (SL_n/H)(K_v)$. 

    Let $g:A' \to A$ be the projection, and consider the commutative diagram:
    \[
    % https://tikzcd.yichuanshen.de/#N4Igdg9gJgpgziAXAbVABwnAlgFyxMJZARgBoAGAXVJADcBDAGwFcYkQAdDgGXoFsARlHoA9AEwBBEQBEQAX1LpMufIRTkK1Ok1bsuvQcPEAKCQHIAlDPmKQGbHgJEyxLQxZtEnDgCEATgAEAMrcAPpgAPQA4jZKDqpEGq407rpeXP7BYZHm8lowUADm8ESgAGZ+EHxIGiA4EEjECuWV1YhiNPVIAMzNIBVVNZ0N7Sk6nt4AHlihAKoAvFzTocAh4dFysf2tjcM9Yx56HMur2RHmm3KUckA
    \begin{tikzcd}
    \Lambda^2(A')^D \arrow[d, "\xi_{SL_n/A'}"] & \Lambda^2A^D \arrow[l, "g^*"] \arrow[d, "\xi_U=\xi_{SL_n/H}"] \\
    \Br (SL_n/A')                                & \Br (SL_n/H) \arrow[l]                                  
    \end{tikzcd}.
    \]
    Let $\bar{a}_1=f(\iota), \bar{a}_2=f(\phi)$, $a_i$ be the image of $\bar{a}_i$ in $A$. Note that $\Lambda^2A'$ is generated by $\bar{a}_1 \wedge \bar{a}_2$, and $g^*\alpha(\bar{a}_1 \wedge \bar{a}_2)=\alpha(a_1 \wedge a_2)$. Since $\bar{a}_1$ and $\bar{a}_2$ commute and $\alpha \in (\Lambda^2 A^D)_{bic}$, $g^*\alpha$ is $0$.
    It follows that $\xi_U(\alpha)$ maps to $0$ in $\Br (SL_n/A')$, and thus specializes to $0$ at all points lying in the image of $(SL_n/A')(K_v) \to (SL_n/H)(K_v)$. In particular, $\xi_U(\alpha)(P)=0 \in \Br K_v$. Since $P$ and $v \nmid p$ were arbitrary, this means that $\xi_U(\alpha)$ is unramified by a well-known consequence of Harari's formal lemma \cite[Thm.\ 2.1.1]{Harari94}.

    The case of a general $k$ follows from the number field case and a ``no-name lemma'' argument. Namely, let $L \coloneqq \Q(\mu_p) \subseteq k$, $\rho: H \hookrightarrow SL_n(k)$ be the representation defining the quotient $SL_{n,k}/H$, and $\rho': H  \hookrightarrow SL_{n'}(L)$ be another faithful representation. Let $U'_k \coloneqq SL_{n',k}/H$, $U'_L \coloneqq SL_{n',L}/H$ (note that we have a natural map $U'_k \to U'_L$), and $U''_k$ be the diagonal quotient $(SL_{n,k} \times SL_{n',k})/H$.
    The variety $U''_k$  is both an $SL_{n,k}$-torsor over $U'_k$ and a $SL_{n',k}$-torsor over $U_k=U$:
    \[
    %% https://tikzcd.yichuanshen.de/#N4Igdg9gJgpgziAXAbVABwnAlgFyxMJZARgBoAGAXVJADcBDAGwFcYkQBVAci4H0BrEAF9S6TLnyEUAJgrU6TVu24DhokBmx4CRcnJoMWbRJ1VD5MKAHN4RUADMAThAC2SPSBwQkshUfYAygAyvMBgXKT8QiA0jPQARjCMAAri2lIgjlhWABY4ag7ObogeXkhkfkomwaFgkdHmQkA
    \begin{tikzcd}
    U_k & U''_k \arrow[l, "{SL_{n',k}}"'] \arrow[r, "{SL_{n,k}}"] & U'_k
    \end{tikzcd}
    \]
    We get a commutative diagram:
    \[
    % https://tikzcd.yichuanshen.de/#N4Igdg9gJgpgziAXAbVABwnAlgFyxMJZARgBoAGAXVJADcBDAGwFcYkQAdDgGXoFsARlHoA9AEwBBEQBEQAX1LpMufIRTlSxanSat2XAEIAnAAQBVAOQB9ANbzFIDNjwEiZLTQYs2iTh2PmFtZ2CkrOqkRimtpeer6Gpma28towUADm8ESgAGZGEHxIGiA4EEhkOt76HAAeWFbAlrZyIDSM9AIwjAAKyi5qIFhg2LD2ufmFiMWlSFGVcX51DZbBLW0dXb3hrr5DI2yhIHkFRTQziADMnro+i-WNza0g7Z09fRG+jDA5OGNHE+UzmVEHNYj4wMxGIx1q8tiods9vr9DsdJlcSsDQTckBCoTDNu8EV8fik5EA
    \begin{tikzcd}
                       & \Lambda^2A^D \arrow[ld, "\xi_{U'_k}" description] \arrow[d, "\xi_{U''_k}" description] \arrow[rd, "\xi_{U_k}"] &                   \\
    \Br U'_k \arrow[r] & \Br U''_k                                                                                                      & \Br U_k \arrow[l]
    \end{tikzcd}
    \]
    where the horizontal maps are pullbacks. Moreover, the morphism $\xi_{U'_k}$ factors as $\Lambda^2 A^D \to[\xi_{U'_L}] \Br U'_L \to \Br U'_k$, where the last morphism is an extension of scalars. We know by the number field case that $\xi_{U'_L}(\alpha) \in {\Br_{ur}U'_L}$, and thus $\xi_{U'_k}(\alpha) \in{\Br_{ur}U'_k}$ and $\xi_{U''_k}(\alpha) \in{\Br_{ur}U''_k}$. Therefore, letting $\beta \coloneqq \xi_{U_k}(\alpha)$, and $\pi$ denote the projection $U''_k \to U_k$, we have that $\pi^* \beta$  belongs to $\Br_{ur} U''_k$. But $U''_k \to U'_k$ is an $SL_{n,k}$-torsor, and in particular possesses a rational section $s: U_k \dashrightarrow U''_k$ by Hilbert 90. It follows that $\beta= (\pi s)^*\beta=s^*(\pi^*\beta)$ is unramified as well, as wished.

\end{proof}

\begin{remark}\label{Rmk:refinement}
    The composition map $\xi_U$ is easily seen to be equal to the composition:
    \[
    \Lambda^2A^D \to[\cup] H^2(A,\Z/p\Z) = H^2(A, \mu_p) \to[\text{inf}] H^2(\Gamma_A, \mu_p) \to H^2(\Gamma_A, \oK^*) \to[\check{C}_{\mathfrak{l}}]  \Br U,
    \]
    thus $\im \xi_U \subseteq \Br_{\lambda}^{ram}U$.
\end{remark}

\subsection{A decomposition of $H^2(A,B)$}\label{SSecSplittingH2}

For general abstract commutative groups $A$ and $B$, we have a split exact sequence \cite[V.6.5]{Brown}:
\begin{equation}\label{Eq:SES}
    0 \to \operatorname{Ext}(A,B) \to H^2(A,B) \to[\omega_B] \Hom(\Lambda^2A,B) \to 0.
\end{equation}

The splitting is non-canonical in general, but it is canonical when $\#B$ is odd: in this case a section $s_B$ of $\omega_B$ is given by $s_B: \gamma \mapsto \left(a_1,a_2 \mapsto \frac{1}{2}\gamma(a_1\wedge a_2)\right)$. If $B=\Z/n\Z$, then $\Hom(\Lambda^2A,\Z/n\Z) = \Lambda^2\Hom(A,\Z/n\Z)$ and $2\cdot s_B$ is the cup-product:
\begin{equation}\label{Eq:sBcupproduct}
    2\cdot s_B= \cup: \Lambda^2\Hom(A,\Z/n\Z) \to H^2(A,\Z/n\Z).
\end{equation}

The splitting $s_B$ of \eqref{Eq:SES} induces a canonical decomposition $H^2(A,B)= \operatorname{Ext}(A,B) \oplus \Hom(\Lambda^2A,B)$.

\begin{lemma}\label{Lem:Hnilpotent}
    Let $p$ be an odd prime number, and assume that $A$ and $B$ have exponent $p$. Let $1 \to B \to H \to[\pi] A \to 1$ be a nilpotent central extension of exponent $p$. Then the class $[H] \in H^2(A,B)$ lies in the image of $s_B$.
\end{lemma}
\begin{proof}
    Equivalently, we have to prove that $[H]$ maps to $0$ under the projection $H^2(A,B) \to {\operatorname{Ext}(A,B)}$. This projection is natural in $A$ (and $B$ as well, but we do not need this). 
    Therefore, since the functor ${\operatorname{Ext}(A,B)}$ is additive in $A$, and every abelian group decomposes into a direct sum of cyclic groups, it suffices to prove the result for a cyclic $A$. If $A$ is cyclic and generated by $a \in A$, every $h \in \pi^{-1}(a)$ must be of order $p$ because $H$ has exponent $p$. The element $h$  thus provides a section $a \mapsto h$ of $\pi$, and $H=B \oplus A$. Hence $[H]=0 \in H^2(A,B)$, and in particular $H$ maps to $0$ in ${\operatorname{Ext}(A,B)}$ as wished.
\end{proof}

\subsection{Formula for the Brauer--Manin pairing}

We work in the setting of \ref{Sec:GBsp}, but we also assume that $k=K$ is a number field. All cohomology and homomorphism groups appearing in this subsection have a natural $\F_p$-vector space structure, so whenever we take tensor products, alternating products, etc, we mean that these operations are performed as $\F_p$-vector spaces.

Let $v$ be a finite place of $K$, $\Gamma_v = \Gal(\oK_v/K_v)$, and $G_v= \Gamma_v^{ab}/p\Gamma_v^{ab}$. Remembering that we fixed an isomorphism $\mu_p \cong \Z/p\Z$, and using the identifications $H^1(\Gamma_v, \F_p)=\Hom(\Gamma_v, \F_p)=\Hom(G_v, \F_p)$, we may identify the perfect local duality pairing:
\[
\cup: H^1(\Gamma_v, \F_p)^{\otimes 2} \to H^2(\Gamma_v, \F_p) \cong \F_p,
\]
with a pairing:
\[
H_v: \Lambda^2 \Hom(G_v,\F_p) \to \F_p.
\]
(Note that this is alternating because $p$ is odd.)

We choose a basis $g_1,\ldots,g_r$ of $G_v$, and a basis $t_1,\ldots,t_a$ of $A$. 
These induce an identification:
\begin{align}
    \label{Id1} \Hom(G_v,A)&= Mat_{a \times r}(\F_p) 
\end{align}
For a finite-dimensional vector space $V/\F_p$, a basis $v_1,\ldots,v_l$ of $V$ induces an identification:
\[
\Lambda^2(V^D) \cong Mat_{ant, l \times l}(\F_p), \ \ \ v_i^D \wedge v_j^D \mapsto e_ie_j^T - e_je_i^T,
\]
where $e_1,\ldots,e_l$ denotes the standard basis of $\F_p^l$, and $Mat_{ant}$ denotes antisymmetric matrices. Specializing to $V=A$ and $V=G_v$, we get identifications 
\begin{align}
    \label{Id2} \Lambda^2A^D&= Mat_{a \times a, ant}(\F_p), \\
     \label{Id3} \Lambda^2G_v^D&= Mat_{r \times r, ant}(\F_p), 
\end{align}

%For $x \in \Hom(G_v,A)$ (resp.\ $x \in \Lambda^2A^D, \Lambda^2G_v^D$), we denote by $M(x)$ its image in $Mat_{a \times r}(\F_p)$ (resp.\ $Mat_{a \times a, ant}(\F_p), Mat_{r \times r, ant}(\F_p)$).

Let $\phi \in \Hom(G_v,A), \beta \in \Lambda^2A^D, \gamma \in \Lambda^2G_v^D$, and let $M_{\phi} \in Mat_{a \times r}(\F_p), M_{\beta} \in Mat_{a \times a, ant}(\F_p), M_{\gamma} \in Mat_{r \times r, ant}(\F_p)$ be their corresponding matrices under the identifications above. Let $\tilde H_v \in Mat_{r \times r, ant}(\F_p)$ be the matrix defined by $(\tilde H_v)_{i,j}=\frac12H_v(g_j \wedge g_i)$.
One easily verifies that:
\begin{equation}\label{Eq:Hvtrace}
    \begin{matrix}
    &\phi^*\beta \in \Lambda^2G_v^D&\text{ corresponds to}&{M_{\phi}}^TM_{\beta} {M_{\phi}} & \in Mat_{r \times r, ant}(\F_p),\\
     \text{ and} &H_v(\gamma)&=&\tr(\tilde{H_v}M_{\gamma}) & \in \F_p.
    \end{matrix} 
\end{equation}
Recall that we have a homomorphism
\[
\xi_U: \Lambda^2\Hom(A,\F_p) \to[\cup] H^2(A, \F_p) \to \Br U.
\]
We may compute the local Brauer pairing between $\im \xi_U$ and $U(K_v)$ as follows:
\begin{lemma}\label{Lem:contopairing}
    Let $\beta \in \Lambda^2\Hom(A,\F_p)$, $b \coloneqq \xi_U(\beta) \in \Br U$, $P \in U(K_v)$, and $\phi \in \Hom(G_v,A)$ be induced by the torsor type $[\lambda|_{P}] \in H^1(\Gamma_v,A) = \Hom(\Gamma_v,A) \cong \Hom(G_v,A)$. Then:
    \[
    (b,P)_{v}=\tr(\tilde{H_v}\tilde{M_{\phi}}^TM_{\beta}M_{\phi}).
    \]
%  The following diagram commutes:
%    \[
%   % https://tikzcd.yichuanshen.de/#N4Igdg9gJgpgziAXAbVABwnAlgFyxMJZABgBoBGAXVJADcBDAGwFcYkQAdDgIQCcACAKogAvqXSZc+QijLFqdJq3ZcAMvQC2AIyj0AegCYAgnoAio8SAzY8BIgYoKGLNohCCAFAGkA+rQCUFhI20kQAzI40zspuXACOAF5BVpK2MsgO8lFKriAAEhAaHlwA4poa9H6kRoFiwVJ2KOSkWYouKhx4GvDJ1g3pzVTZ7bGdWN0IdSkhjcgRrdG5XABaAPRoK6IKMFAA5j0ooABmvIVIAOw0OBBIESA49FiM7AAWEBAA1kw4ySdniA57jdEAAWYYxEDILhoF5YAA+PmAAFoRJRfqcNEhmkCkGQ2hCuAAPLA+YRTP6Y0FXYHnckYpCA663VEiIA
%\begin{tikzcd}
%\Lambda^2A^D \arrow[d, "\xi_U"] & \times & {Hom(\Gamma_v,A)} \arrow[r]                 & \Z/p\Z \arrow[d, hook'] \\
%\Br U                           & \times & U(K_v) \arrow[u, "{[\lambda|_{-}]}"] \arrow[r] & \qz                    
%\end{tikzcd}
%    \]
\end{lemma}
\begin{proof}
    This follows by just unravelling the notation. Namely, we have that $(b,P)_v = \inv_v (P^*b)$, and the latter is equal to $H_v(\phi^*\beta)=\inv_v(P^*b)$ by the commutativity of the diagram
    \[
    % https://tikzcd.yichuanshen.de/#N4Igdg9gJgpgziAXAbVABwnAlgFyxMJZAJgBoAGAXVJADcBDAGwFcYkQAJAPWIAoBBUgB0hAMQD6aAJQgAvqXSZc+QigDMFanSat23PgFVhYyTPmLseAkTIBGLQxZtEnHrwDi42sYnS5CkAxLFSINexpHXRd9XhF3egBbBPovH1MAXhiAaVSRXzMAoOVrFFtNCJ1nEBEAGUSAIyh6Hn4uABF-C2LVZDLw7Sd2Woamnk9ads7ApSsegBZygaiQWKEAIQAnAAIDKWQ0Simi2aIF-siq1c2tnNo9g5EAYwIAcy28ySOZkJRyRYuhkIAB5YcQGFzmabBErIP7nSp6LwQrQwKAveBEUAAMw2EASSDIIBwECQGiWVREWDAWKmOLxSD+RJJiDK5MBjwAFjBHgBrYAAYVk4mAIkYI3oslpuPxiEZxIJFUGLhEaA5WC4ACopfSWTR5YgyQDlUJVeqtZC6TKAKx65mEo3VISPZhobUyhZMhmK5ZPF1upAe-U2tnG02a-26z2IABsFulpNtSAA7HGddHE4gk96KSa1eHZJRZEA
    \begin{tikzcd}
    \xi_U: & \Lambda^2A^D \arrow[r, "\cup"] \arrow[d, "\phi^*"] & {H^2(A,\F_p)} \arrow[r, "\check{C}_{\lambda}"] \arrow[d, "\phi^*"] & {H^2(U,\F_p)} \arrow[d, "{P^*}"] \arrow[r]  & {(\Br U)[p]} \arrow[d, "{P^*}"] \\
    H_v:   & \Lambda^2G_v^D \arrow[r, "\cup"]                   & {H^2(G_v,\F_p)} \arrow[r, "\inf"]                                  & {H^2(\Gamma_v,\F_p)=H^2(K_v,\F_p)} \arrow[r] & {(\Br K_v)[p]\cong \F_p}       .
    \end{tikzcd}
    \]
    The formulas \eqref{Eq:Hvtrace} now give the statement.
\end{proof}

\subsection{Proof of Theorem \ref{Thm:gbsp}}
    Let $\Theta_v \subseteq \Hom(\Gamma_v,A)$ be the set of torsor types to which $\lambda$  specializes, i.e.\ the image of $[\lambda|_{-}]:U(K_v) \to \Hom(\Gamma_v,A)$. 
    
\begin{lemma}
    The set $\Theta_v$ is the inverse image of $0$ under:
    \begin{equation}\label{EqNotlinear}
        \Hom(\Gamma_v,A) \to H^2(\Gamma_v,B), \ \ \xi \mapsto \xi^*([H]),
    \end{equation}
    where $[H] \in H^2(A,B)$ is the class representing $H$.
\end{lemma}

 (Note that this inverse image is not a subspace in general!)

\begin{proof}
    The commutative diagram with exact first row:
    \[
    % https://tikzcd.yichuanshen.de/#N4Igdg9gJgpgziAXAbVABwnAlgFyxMJZARgBoAGAXVJADcBDAGwFcYkQAKAZQBkB9MAHoAEgEoOAaT61RAXgCqk6aJABfUuky58hFACYK1Ok1bthAPWJLapMWo0gM2PASIHiRhizaIQFq1I2AIJyADqhwhAAthzhAOL0UVH00qQh9prOOkQAzIY0Xqa+-takvAJy5BmOWi66yOT5xt7s5WDWKqpGMFAA5vBEoABmAE7RSI0gOBBIxOrDY1GzNNNIevMgo+OIk6uIBs1FIMjhaAAWWHzAALRd1VtLiGRTM4g5Gw9IACwrr+RdqiAA
\begin{tikzcd}[column sep=tiny]
SL_n(K_v) \arrow[r] & (SL_n/H)(K_v)=U(K_v) \arrow[r] \arrow[rd, "{[\lambda_{-}]}"] & {H^1(K_v,H)} \arrow[d] \arrow[r] & {H^1(K_v,SL_n)=0} \\
                    &                                                           & {H^1(K_v,A)=\Hom(\Gamma_v,A)}    &                  
\end{tikzcd}
    \]
    shows that $\Theta_v$ is the image of $H^1(K_v,H)=\Hom(\Gamma_v,H) \to \Hom(\Gamma_v,A)$.
    Equivalently, $\Theta_v$ consists of those homomorphisms that lift from $A$ to the central extension $H$.  The statement now follows from the theory of central extensions \cite[Section IV.3]{Brown}.
\end{proof}

\begin{proposition}\label{Prop:example}
    For every $p \geq 5$, number field $K$ with $\mu_p \subseteq K^*$, and place $v$ of $K$ dividing $p$, there exists a constant $c(v) \geq 0$ such that for any $a \geq c(v)$, there exists a metabelian nilpotent $H$ of exponent $p$ with $H^{ab}$ of rank $a$ such that, letting $A=H^{ab}$:
    \begin{enumerate}[label=(\roman*)]
        \item $(\Lambda^2A^D)_{bic}=\Lambda^2A^D$;
        \item there exists $\alpha \in \Lambda^2A^D$ such that the local pairing $(-,\xi_U(\alpha)):U(K_v) \to \Z/p\Z$ attains at least two values.
    \end{enumerate}
\end{proposition}
\begin{proof}
    Let $a$ be a natural number and $A\coloneqq (\Z/p\Z)^a$. Lemma \ref{Lem:setting} of the appendix guarantees the existence of a metabelian nilpotent $H$ of exponent $p$ with $H^{ab}=A$ such that $B=[H,H]$ has rank $b \coloneqq 2a-3$ and $(\Lambda^2A^D)_{bic}=\Lambda^2A^D$.

    Recall that $\Theta_v$ is the image of $[\lambda|_{-}]:U(K_v) \to \Hom(\Gamma_v,A)$. Let $\Xi_v \subseteq \Theta_v$ be the image under $[\lambda|_{-}]$ of the left kernel of $U(K_v) \times \Br_{e,ur} U \to \qz, (P,b) \mapsto \inv_v(P^*b)$. Since $e \in \Xi_v$, $\Xi_v$ is non-empty and to prove point $(ii)$ it suffices to prove that $\Xi_v \neq \Theta_v$. We do so by proving that the cardinality of these two sets have different $p$-adic valuation for $a \geq c(v)$. We start by computing $v_p(\#\Theta_v).$
        
    Recall the identifications
    \[
   \operatorname{Hom}\left(G_v, A\right)=\operatorname{Mat}_{a \times r}\left(\mathbb{F}_p\right),
\Lambda^2 A^D=M a t_{a \times a, a n t}\left(\mathbb{F}_p\right) ,
 \Lambda^2 G_v^D=M a t_{r \times r, a n t}\left(\mathbb{F}_p\right).
    \]
    Identifying $B$ with $\F_p^b$ we get identifications $H^2(A,B)=H^2(A,\F_p)^b$ and $H^2(\Gamma_v,B)=H^2(\Gamma_v,\F_p)^b\xrightarrow[\sim]{\inv_v} \F_p^b$. By Lemma \ref{Lem:Hnilpotent} and \eqref{Eq:sBcupproduct}, there exist $h_1,\ldots,h_b \in \Lambda^2\Hom(A,\F_p)=  Mat_{a \times a,ant}(\F_p)$ such that $[H] = (\mathfrak{h}_1,\ldots,\mathfrak{h}_b) \in H^2(A,\F_p)^b, \ \mathfrak{h}_i \coloneqq\cup(h_i)$, where $\cup$ indicates the cup-product $\Lambda^2\Hom(A,\F_p)=\Lambda^2H^1(A,\F_p) \to H^2(A,\F_p)$. Then by \eqref{EqNotlinear}:
    \begin{equation}
        \Theta_v=\{M \in \Hom(\Gamma_v,\F_p)  \mid \inv_v(M^*\mathfrak{h}_i)=0, \ i=1,\ldots,b\}.
    \end{equation}
    By Lemma \ref{Lem:contopairing}, under the identification $\Hom(\Gamma_v,\F_p)=\Hom(G_v,\F_p)=Mat_{a \times r}\left(\mathbb{F}_p\right)$, $\Theta_v$ corresponds to:
    \begin{equation}\label{Eq:description}
        {\Theta}_v=\{M \in Mat_{a \times r}\left(\mathbb{F}_p\right)  \mid \tr(\tilde{H_v}\tilde{M}^T \tilde{h_i} \tilde{M})=0, \ i=1,\ldots,b\}.
    \end{equation}

    A special case of the Ax-Katz theorem \cite{KA} states that the $p$-adic valuation of the number of solutions of a system of $m$ polynomial equations of degree $\leq d$ in $n$ (affine) variables in $\F_p$ is at least $\left\lceil\frac{n-dm}{d}\right\rceil$.
    Hence, since \eqref{Eq:description} describes $\Theta_v$ as the solution set of $b$ quadratic equations in $ra$ variables:
    \[
    v_p(\#\Theta_v) \geq \left\lceil\frac{ra-2b}{2}\right\rceil = \left\lceil\frac{(r-4)a+6}{2}\right\rceil \geq \left\lceil\frac{2a+6}{2}\right\rceil
    \]
    as $r \geq p+1 \geq 6$ by \cite[Thm.\ 7.5.11]{GermanBook}.

    We now compute $v_p(\# \Xi_v)$, using Lemma \ref{Prop:Isotropic_subspaces} from the appendix. Recall that $(\Lambda^2A^D)_{bic}=\Lambda^2A^D$ by our choice of $H$, and thus $\xi_U(\Lambda^2A^D)=\Br_{ur}U$ by Theorem \ref{ThmBogomolovext}. Hence Lemma \ref{Lem:contopairing} implies that $P$ lies in the left kernel  of the local Brauer pairing $U(K_v) \times \Br_{ur} U \to \qz$ if and only if $H_v(M_P^*\alpha)=(P,\xi_U(\alpha))_v=0$ for all $\alpha \in \Lambda^2A^D$. 
    Therefore, \eqref{Eq:Hvtrace} implies that the image of $\Xi_v$ under the identification $\Hom(\Gamma_v,\F_p)=\Hom(G_v,\F_p)=Mat_{a \times r}\left(\mathbb{F}_p\right)$ is:
    \begin{align*}
        \Xi_v
        &=\{M \in \Theta_v \subseteq Mat_{a \times r}(\F_p)\mid \tr(\tilde{H}_vM^T N M)=0, \text{ for all } N \in Mat_{a \times a,ant}(\F_p)\}  \\
        &=\{M \in Mat_{a \times r}(\F_p)  \mid \tr(\tilde{H}_vM^T N M)=0, \text{ for all } N \in Mat_{a \times a,ant}(\F_p) \},\\
        &=\{M \in Mat_{a \times r}(\F_p)  \mid \tr(M\tilde{H}_vM^T N)=0, \text{ for all } N \in Mat_{a \times a,ant}(\F_p) \},
    \end{align*}
    where the second identity holds because the $b$ quadratic equations describing $\Theta_v$ are redundant in the description of $\Xi_v$ (see \eqref{Eq:description}).  
    Since $\tr(X^TY)$ induces a perfect pairing on $Mat_{a \times a,ant}(\F_p)$ and $M\tilde{H}_vM^T$ lies in $Mat_{a \times a,ant}(\F_p)$, we obtain from the above:
    \[
    \Xi_v^T=\{M \in Mat_{r \times a}(\F_p)  \mid M\tilde{H}_v^TM^T=0\}.
    \]
    The matrix $\tilde{H}_v^T=-\tilde{H}_v$ is invertible by local duality, and so Lemma \ref{Prop:Isotropic_subspaces} shows that there exists a function $C: \N \to \Z$ (depending only on $p$) such that $\# \Xi_v^T\equiv C(r) \mod p^a$ if $a \geq 2r$. Recall that $r=r(v)=[K_v:\Q_p]+2$, and let $c(v)= \max\{2r,v_p(C(r))+1\}.$ If $a \geq c(v)$, then $v_p(\# \Xi_v^T)= v_p(C(r))$ by the ultrametric triangular inequality.
    Now
    \[
    v_p(\# \Xi_v)= v_p(\# \Xi_v^T) = v_p(C(r))<a<\left\lceil\frac{2a+6}{2}\right\rceil\leq v_p(\#\Theta_v)
    \]
    as wished.
\end{proof}

We refer the reader to the proof of Lemma \ref{Prop:Isotropic_subspaces} for a formula for $C(r)$.
As a consequence of Proposition \ref{Prop:example}, we may now prove Theorem \ref{Thm:gbsp}:
\begin{proof}[Proof of Theorem \ref{Thm:gbsp}]
    Let $H$ be any group as in Lemma \ref{Prop:Isotropic_subspaces}, and let $U=SL_{n,K}/H$. By point $(ii)$ of the lemma there is a place $v$, an element $\beta \in \Lambda^2A^D$, and points $P_v \neq Q_v \in U(K_v)$ such that:
    \[
    (P_v,b)_v \neq (Q_v,b)_v \in \qz, \ \ b\coloneqq \xi_U(\beta) \in \Br X.
    \]
    Consider then the adelic point $\underline{P} \in U(\A_K)$ (resp. $\underline Q$) that is equal to $e$ at all places $\neq v$ and is equal to $P_v$ (resp.\ $Q_v$) at $v$. Then $(\underline{P},b)_{BM}=\inv_v(b(P_v)) \neq \inv_v(b(Q_v))=(\underline{Q},b)_{BM}$. Hence (at least) one between $\underline P$ and $\underline Q$ does not lie in $X(\A_K)^{\Br X}$, concluding the proof.
\end{proof}

\appendix

\section{Elementary counting facts}\label{Sec:Elementary}

\begin{lemma}\label{Prop:Too_little_points}
	Let $n \leq N$ be positive integers and $X \subseteq \P^N(\F_p)$ a subset of cardinality $< \#\P^n(\F_p)$. There exists then an $n$-codimensional subspace $L \subseteq \P_{\F_p}^N$ such that $X \cap L =\emptyset$.
\end{lemma}
\begin{proof}
	Let $k\geq 0$ be the smallest integer such that $X$ intersects every $k$-dimensional subspace in $\P^N_{\F_p}$. If $k=0$, there is nothing to prove. Otherwise, let $L \subseteq \P^N_{\F_p}$ be a $(k-1)$-dimensional subspace such that $L \cap X =\emptyset$. Let $\pi_L:\P^N\setminus l \rightarrow \P^{N-k}$ be a projection outside of $L$. We know by assumption that $\pi_L(X(\F_p))=\P^{N-k}(\F_p)$, hence $\#X(\F_p) \geq \#\P^{N-k}(\F_p) \Rightarrow N-k < n$, i.e. $k \geq N-n+1$. Hence the dimension of $L$ is $\geq N-n$ and it is the sought subspace.
\end{proof}

The following lemma is inspired by \cite[Sec.\  5]{BogomolovMumbai}, see also \cite[p.\ 37]{mumbai04}.

\begin{lemma}\label{Lem:setting}
	Let $p\neq 2$ be a prime. For every $\F_p$-vector space $A$ of dimension $4 \leq a < \infty$, there exists an $\F_p$-vector space $B$, of dimension $b=2a-3$, and a (surjective) morphism 
	\begin{equation}\label{star}
	\mathfrak{c}:\Lambda^2A \rightarrow B,
	\end{equation}	
	such that, if $1 \rightarrow B \rightarrow H \to[\pi] A \rightarrow 1$ is the extension whose commutator map is $\mathfrak{c}$, we have that there are no pure wedges $0 \neq a_1 \wedge a_2 \in \Lambda^2A$ lying in the kernel of $\mathfrak c$.
\end{lemma}

\begin{proof}
	Let $X \subseteq \P_{\F_p}(\Lambda^2A)$ be the image of the ``alternating Segre morphism'' 
	$$
	- \wedge-:\P(A) \times \P(A) \setminus \Delta \rightarrow \P(\Lambda^2A), 
	$$
	which is isomorphic to the Grassmanian variety $\Gr_{\F_p}(2,A)$. Since $X(\F_p)$ parametrizes two-dimensional $\F_p$-subspaces of $A$:
	\begin{equation}\label{Eq:number_points_Segre}
	\#X(\F_p)= 	\frac{(p^a-1)(p^{a-1}-1)}{(p^2-1)(p-1)}.
	\end{equation}
	
	It suffices to show that there exists a $(2a-3)$-codimensional subspace $L$ in $\P(\Lambda^2A)$ such that $L \cap X(\F_p) =\emptyset$, and choose $\mathfrak{c}$ such that $\Lambda^2A \supseteq \F_p \cdot L(\F_p)= \Ker \mathfrak{c}$. 
	Noting that:
	\[
	\frac{(p^{a}-1)(p^{a-1}-1)}{(p^2-1)(p-1)}<  \frac{(p^{a}-1)(p^{a-1}-1)}{(p+1)(p-1)} \leq\frac{(p^{2a-2}-1)(p+1)}{(p+1)(p-1)}= \# \P^{2a-3}(\F_p),
	\]
	such a subspace always exists by Lemma \ref{Prop:Too_little_points}.
\end{proof}

\begin{lemma}\label{Prop:Isotropic_subspaces}
	Let $A,V$ be $\F_p$-vector spaces with $p \neq 2$, and let $a := \dim A, r:= \dim V$. Assume that $a \geq 2r$, and that $V$ is endowed with an alternating non-degenerate bilinear form $b:V \times V \to \F_p$. We then have that:
	\[
	\Xi(A,V) \defeq \#\{\xi \in \Hom(A,V) \ \mid \ \xi^*b=0 \} \equiv C(r) \mod p^a,
	\] 
	where $C(r)$ is a non-zero integer depending only on $r$.
\end{lemma}
\begin{proof}
	Let
	\[
	M_d \defeq \#\{\text{isotropic $d$-dimensional subspaces in $V$}\},
	\]
	\[
	I_d \defeq \#\{\text{surjective homomorphisms from $A$ to a $d$-dimensional $\F_p$-vector space}\}.
	\]
	We then have that:
	\[
	\Xi(A,V)=\sum_{d=0}^{\min(a,\frac{r}{2})}I_dM_d
	\]
	(note that the fact that $V$ is endowed with a non-degenerate alternating linear form and $p \neq 2$ implies that $r$ is even).
	One can easily see that:
	\[
	I_d=(p^a-1)\cdot (p^a-p) \cdots (p^a-p^{d-1}), \ \text{for every $d \leq a$},
	\]
	\[
	M_d=\frac{(p^r-1)\cdot (p^{r-1}-p)\cdots (p^{r-d+1}-p^{d-1})}{(p^d-1)\cdot (p^d-p)\cdots (p^d-p^{d-1})}, \ \text{for every $d \leq r/2$}.
	\]
	In particular, $\Xi(A,V)=\Xi'(a,r)$ depends only on $a$ and $r$. Note that, for a fixed $r$, $\Xi'(a,r)$ converges $p$-adically, as $a \to \infty$, to the following sum (which happens to be an integer number):
	\begin{align*}
		C(r) \defeq\Xi'(\infty,r)&\defeq \sum_{d=0}^{r/2}(-1)^{d}\frac{(p^r-1)\cdot (p^{r-1}-p)\cdots (p^{r-d+1}-p^{d-1})}{(p-1)\cdot (p^{2}-1)\cdots (p^d-1)}\\&=\sum_{d=0}^{r/2}(-1)^{d}\binom{r/2}{d}_{p^2} (p+1)\cdots (p^d+1),
	\end{align*}
	where the subscript in the binomial denotes a Gaussian binomial coefficient (an integer number).
	Moreover, $\Xi'(a,r) \equiv \Xi'(\infty,r) \mod p^a$ if $a \geq 2r$.
	Denoting by $a(d)$ the term multiplying the $(-1)^d$ appearing above, we notice that the sequence $a(0) ,\ldots, a(r/2)$ is strictly increasing, as follows by induction from the fact that $\frac{p^{r-d+1}-p^{d-1}}{p^d-1}>1$ for all $d \in \{0,\ldots,r/2\}$. In particular, a standard elementary calculus argument (\`a la Leibniz' rule) shows that $\Xi'(\infty,r) \neq 0$.
\end{proof}

\section{Other works where ramified descent appears}\label{Sec:Other}

Let me mention other works where the idea of ``ramified descent'' has already appeared. One is \cite{HS} by Harpaz and Skorobogatov (successor to Skorobogatov and Swinnerton-Dyer's work \cite{SSD} \cite{SDsolubility}), where the authors use the cyclic ramified covers of some specific Kummer surfaces to prove that, under certain technical assumptions, these satisfy the Hasse principle.

Another work is Corvin and Schlank's paper \cite{CS}, where the authors build upon Poonen's example \cite{PoonenInsufficiency} to show (employing one specific ramified $S_4$-cover) that the following obstruction is \emph{stronger} than \'etale-Brauer-Manin obstruction:
\[
X(\A_K)^{\Br,ram,sol}=\bigcap_{\substack{\psi:Y \to X \\ G-\text{cover} \\ G \text{ solvable}}} \overline{\bigcup_{\xi \in H ^1(K,G)} \psi'_{\xi}(Y_{\xi}^{sm}(\A_K))^{\Br Y_{\xi}^{sm} }},
\]
where the $\psi'_{\xi}$ is the composition $Y_{\xi}^{sm} \to Y_{\xi} \to[\psi_{\xi}] X$.

Lastly, we mention Sections 11.5 and 14.2.5 of Colliot-Th\'el\`ene and Skorobogatov's book \cite{BGbook}, where ramified descent is investigated for $\mu_n$-covers. In particular, in Theorem 14.2.25 of \emph{loc.cit.}, the authors prove a result which translates in our language to saying that, if $\lambda:V \to U$ is a $\mu_n$-torsor  such that there is a divisor on $X$ over which the ``compactification'' $\psi:Y \to X$ of $\lambda$ (notation as in Section \ref{Sec:Setting}) is totally ramified, then $X(\A_K)^{\lambda}=X(\A_K)^{\Br_{\lambda}^{ram}X}$. Their result and our Proposition \ref{Prop2} naturally lead us to the question:

\begin{question}\label{QuestionB1}
    Let $\lambda:V \to U$ and $\psi:Y \to X$ be as in Section \ref{Sec:Setting}. Assume that the cover $Y \to X$ is totally ramified, i.e. $Y$ is geometrically integral and $Y \rightarrow X$ does not have any unramified subcovers. Does one then have that $X\left(\mathbb{A}_K\right)^\lambda=X\left(\mathbb{A}_K\right)^{\mathrm{Br}_\lambda^{ram} X}$?
\end{question}

Note that a positive answer to the question above would guarantee that, for instance, if $Y$ is a variety all of whose $G$-twists satisfy the Hasse principle, then $X$ satisfies the Hasse principle up to Brauer-Manin obstruction.

Let us mention that, when $G$ is supersolvable and $Y$ is rationally connected, 
Harpaz and Wittenberg prove that \cite[Theorem 1.4]{HW20}:
\[
X\left(\A_K\right)^{\mathrm{Br} X }=\overline{\bigcup_{\xi \in H^1(K,G)} \psi^{\sm}_{\xi}\left(Y_{\xi}\left(\A_K\right)^{\mathrm{Br} Y^{\sm}_{\xi}}\right)}
\]
(using our notation). It follows that $X\left(\mathbb{A}_K\right)^\lambda \supset X\left(\mathbb{A}_K\right)^{\mathrm{Br} X}$, i.e.\ Brauer--Manin obstruction is finer than ramified descent obstruction, but it also seems likely that their methods could be in fact used to give a positive answer to Question \ref{QuestionB1} in this case. For instance, when, in addition to the conditions above, $\oK[V]^*/\oK^*=0$ and $\Pic \overline{V}=0$ (e.g.\ if $V=SL_n$), then $\mathrm{Br}_\lambda^{ram} X=\Br X$ by Remark \ref{Rmk:Piceq0}, and a positive answer to the question follows.

\begin{acknowledgement}
	This work was part of the author's thesis.
        The author thanks his PhD advisor David Harari, for providing motivation towards the question, for his guidance and for his helpful comments on the text. 
        Moreover, the author thanks his thesis examiners, Olivier Wittenberg and Daniel Loughran, for their interesting comments. 
        The author especially thanks Olivier Wittenberg for pointing out that the counterexamples of the last section provide an interesting example of transcendental obstruction to the Grunwald Problem, for having provided a simplification of the proof of Theorem 1.1, and for 
        making him notice Theorem \ref{Thm:BrSLnG}. 
        The author thanks the anonymous referees who helped greatly improve the exposition of the paper, and provided valuable insights into simplifications of various proofs, in particular what is now Proposition \ref{Prop2} and the general proof path of Theorem \ref{Thm:gbsp} in Section \ref{Sec6}. 
        Part of this work was written at the Max Planck Institute for Mathematics at Bonn, and the author is grateful for their financial support and for the excellent working conditions they provided.
\end{acknowledgement}

\end{document}